\documentclass[reqno, oneside, 12pt]{amsart}

\usepackage[letterpaper]{geometry}
\usepackage{flafter} 
\usepackage{float}
\usepackage{appendix}
\usepackage{amsmath}
\usepackage{amssymb}
\usepackage{enumerate}
\usepackage{verbatim}
\usepackage{mathrsfs}
\usepackage{mathtools}
\usepackage{hyperref}
\usepackage{amsthm}
\usepackage{placeins}
\usepackage{tikz-cd}
\newcommand\numberthis{\addtocounter{equation}{1}\tag{\theequation}}

\DeclareMathOperator{\Jac}{Jac}
\DeclareMathOperator{\USp}{USp}
\DeclareMathOperator{\Sp}{Sp}
\DeclareMathOperator{\U}{U}
\DeclareMathOperator{\GL}{GL}
\DeclareMathOperator{\ST}{ST}
\DeclareMathOperator{\Gal}{Gal}
\DeclareMathOperator{\End}{End}
\DeclareMathOperator{\diag}{diag}
\DeclareMathOperator{\Span}{Span}
\DeclareMathOperator{\Aut}{Aut}
\DeclareMathOperator{\Zar}{Zar}
\DeclareMathOperator{\Frob}{Frob}
\DeclareMathOperator{\Tr}{Tr}
\DeclareMathOperator{\ord}{ord}

\newtheorem{theorem}{Theorem}[section]
\newtheorem{example}[theorem]{Example}
\newtheorem{proposition}[theorem]{Proposition}
\newtheorem{claim}[theorem]{Claim}
\newtheorem{lemma}[theorem]{Lemma}
\newtheorem{conjecture}[theorem]{Conjecture}
\newtheorem{definition}[theorem]{Definition}
\newtheorem{corollary}[theorem]{Corollary}
\newtheorem{algorithm}[theorem]{Algorithm}

\theoremstyle{remark}
\newtheorem{remark}{Remark}
\newtheorem*{remark*}{Remark}

\author{Melissa Emory }
\address{Department of Mathematics, University of Toronto; 40 St. George Street, Toronto, Ontario, Canada M5S 2E4}
\email{memory@math.toronto.edu}

\author{Heidi Goodson}
\address{Department of Mathematics, Brooklyn College; 2900 Bedford Avenue, Brooklyn, NY 11210 USA}
\email{heidi.goodson@brooklyn.cuny.edu}

\author{Alexandre Peyrot}
\email{peyrotalexandre12@gmail.com}

\title[Towards the Sato-Tate Groups of Trinomial Hyperelliptic Curves]{Towards the Sato-Tate Groups of Trinomial Hyperelliptic Curves}

\begin{document}

\begin{abstract}
We consider the identity component of the Sato-Tate group of the Jacobian of curves of the form
$$C_1\colon y^2=x^{2g+2}+c,  C_2\colon y^2=x^{2g+1}+cx,  C_3\colon y^2=x^{2g+1} +c,$$
where $g$ is the genus of the curve and $c\in\mathbb Q^*$ is constant.  

We approach this problem in three ways. First we use a theorem of Kani-Rosen to determine the splitting of Jacobians for $C_1$ curves of genus 4 and 5 and prove what the identity component of the Sato-Tate group is in each case. We then determine the splitting of Jacobians of higher genus $C_1$ curves by finding maps to lower genus curves and then computing pullbacks of differential 1-forms. In using this method, we are able to relate the Jacobians of curves of the form $C_1$, $C_2$, and $C_3$. Finally, we develop a new method for computing the identity component of the Sato-Tate groups of the Jacobians of the  three families of curves. We use this method to compute many explicit examples, and find surprising patterns in the shapes of the identity components $\ST^0(C)$ for these families of curves.  
\end{abstract}



\maketitle
\section{Introduction}\label{sec:intro}

Let $C$ be a smooth projective curve defined over $\mathbb Q$. For primes $p$ of good reduction, we define the trace of Frobenius to be
$$t_p(C)=p+1-\#\overline{C}(\mathbb F_p)$$
where $\overline{C}$ denotes the reduction of $C$ modulo $p$. A theorem of Weil \cite{Weil1948} gives the following bound for the trace of Frobenius 
$$|t_p|\leq 2g\sqrt{p},$$
where $g$ is the genus of the curve.

Let $x_p=t_p/\sqrt{p}$ denote the normalized trace. Then the Weil bounds tell us that $x_p\in[-2g,2g]$, and we can look at the distribution of the $x_p$ in this interval as $p\rightarrow \infty$. This distribution is known for elliptic curves. The values are not uniformly distributed over the interval $[-2,2]$, though they do have a predictable limiting pattern. In the 1960s, Sato and Tate independently conjectured that, for elliptic curves defined over $\mathbb Q$ without complex multiplication, the normalized traces are equidistributed with respect to the measure $\frac{1}{2\pi}{\sqrt{4-x^2}} dx$. Barnet-Lamb, Geraghty, Harris, Shepard-Barron, and Taylor \cite{Barnet2011, Harris2010} recently proved the Sato-Tate Conjecture for elliptic curves without complex multiplication. As is often the case with celebrated results in number theory, proving the Sato-Tate Conjecture required heavy machinery and merged three massive mathematical theories: $L$-functions, automorphic forms, and Galois representations.

Traces of  higher genus curves are expected to have Sato-Tate-like distributions. To determine the distributions, we study the Sato-Tate group of the Jacobians of the curves.  Recall that the Jacobian of a genus $g$ curve is an abelian variety of dimension $g$. Associated to any abelian variety of dimension $g$ over a number field  there is a compact subgroup of $\USp(2g)$ known as the Sato-Tate group (see \cite[Section 3.2]{SutherlandAWSNotes}) that is uniquely determined up to conjugacy and comes equipped with a map that sends Frobenius elements to conjugacy classes with the appropriate normalized trace.  It is conjectured that if we order Frobenius elements by norm, this sequence of conjugacy classes is equidistributed with respect to the push forward of the Haar measure on the Sato-Tate group, and this can be viewed as a generalization of the Sato-Tate conjecture (see, for example, \cite[Section 3.3]{SutherlandAWSNotes}).

Determining these Sato-Tate groups is the source of ongoing work. For example, Fit\'e, Kedlaya, Rotger, and Sutherland \cite{Fite2012} determine the complete set of Sato-Tate groups that arise for abelian surfaces over number fields. In \cite{Fite2016}, Fit\'e and Sutherland give the Sato-Tate groups and distributions for the following families of genus 3 hyperelliptic curves
$$y^2=x^8+c \; \text{ and } \; y^2=x^7-cx,$$
where $c\in\mathbb Q^*$ is constant. Fit\'e, Lorenzo Garc\'ia, and Sutherland have also worked out the Sato-Tate groups for other genus 3 curves (see \cite{Fite2018}). In \cite{Arora2016}, Arora, et al. prove a generalized Sato-Tate conjecture for $\mathbb Q$-twists of the genus 3 curve $y^2=x^8-14x^4+1$. See \cite{fit2019satotate} for an in-depth discussion of the Sato-Tate groups of abelian varieties of dimension $3$. 

In this paper, we extend this work to families of hyperelliptic curves over $\mathbb Q$ of the form
$$C_1\colon y^2=x^{2g+2}+c, \quad C_2\colon y^2=x^{2g+1}+cx, \quad C_3\colon y^2=x^{2g+1} +c,$$ 
where $g$ is the genus of the curve and $c\in\mathbb Q^*$ is a constant.  We denote the Sato-Tate group of the Jacobian of a smooth projective curve   by $\ST(C):=\ST(\Jac(C)_{\mathbb Q})$.  Note that while the Sato-Tate group is a compact Lie group, it may not be connected \cite{Fite2016}. In our work we study the connected component of the  identity  of $\ST(C)$, denoted $\ST^0(C):=\ST^0(\Jac(C)_{\mathbb Q})$. Note that $\ST^0(C)$ is isomorphic to the full Sato-Tate group $\ST(\Jac(C)_{F})$, where $F$ is the minimal extension over which all endomorphisms of  $\Jac(C)$ are defined.

This problem of determining the identity component of the Sato-Tate groups of families of trinomial hyperelliptic curves was originally posed as part of the Arizona Winter School  {\it Analytic Methods in Arithmetic Geometry} in March 2016. Using similar methods to \cite{Fite2016} and a theorem of  Kani-Rosen \cite[Theorem C]{KaniRosen1989}, we obtain the following explicit results for $\ST^0(C)$ for families of genus 4 and 5 curves (the notation is defined in Section \ref{sec:background}). 
\begin{theorem}\label{genus4}
The identity component of the Sato-Tate group of the Jacobian of the hyperelliptic curve $y^2=x^{10}+c$ is $\U(1)_2 \times \U(1)_2.$
\end{theorem}

\begin{theorem}\label{genus5} 
The identity component of the Sato-Tate group of the Jacobian of the hyperelliptic curve $y^2=x^{12}+c$ is $\U(1)_2 \times \U(1)_3.$
\end{theorem} 
These results are proved in Section \ref{sec:genus45proof}. The methods used for the proof of Theorem \ref{genus4} and Theorem \ref{genus5}  require using automorphisms and morphisms of curves to prove the result.   To generalize results like those of Theorem \ref{genus4} and Theorem \ref{genus5} to higher genus $C_1$ curves, we prove a partial splitting of the Jacobians of higher genus curves in the following theorem (see Theorem \ref{splitting}).

\begin{theorem}\label{splittheorem}
Let $v_2\colon  \mathbb Q^* \rightarrow \mathbb Z$ denote the $2$-adic valuation map, i.e. $v_2(a/b) = \alpha$, where $\frac{a}{b}=2^\alpha\frac{e}{d}$ and $p$ does not divide $e$ or $d$. Let $C_1\colon  y^2= x^{2g+2}+c$ be a hyperelliptic curve of genus $g$ and write $k:= v_2(g+1)$. Then we have the following isogeny over $\overline{\mathbb Q}$.
$$\Jac(C_1) \sim \Jac(y^2= x^{{(g+1)}/{2^k}}+c)^2 \times \prod_{i=0}^{k-1} \Jac(y^2=x^{{(g+1)}/{2^{i}}+1}+cx),$$
which relates the curves
$$C_1\colon y^2=x^{2g+2}+c, \quad C_2\colon y^2=x^{2g+1}+cx, \quad C_3\colon y^2=x^{2g+1} +c.$$ 
\end{theorem}
Theorem \ref{splittheorem}  breaks down the Jacobian of a curve into the Jacobians of lower genus curves. We break these Jacobians down even further in Section \ref{sec:breakdownjacobian}. In some cases, we can then use known results for lower genus curves (see, for example, \cite{Barnet2011, Fite2012, Fite2016}) to immediately determine the identity component of the Sato-Tate group.  Also, note that Theorems \ref{genus4} and \ref{genus5} follow as corollaries to Theorem \ref{splittheorem}.

In Section  \ref{sec:altmethod}, we describe a new algorithm that computes the identity component of the Sato-Tate group of the Jacobian of hyperelliptic curves $C_1, C_2,$ and $C_3$ mentioned above. 
\begin{theorem}\label{thm:alg}
Algorithm \ref{algorithm} gives the identity component of the Sato-Tate group of the Jacobian of curves of the form 
$$C_1\colon y^2=x^{2g+2}+c, \quad C_2\colon y^2=x^{2g+1}+cx, \quad C_3\colon y^2=x^{2g+1} +c.$$ 
\end{theorem}

Using  Theorem \ref{thm:alg}  we prove the following (see Theorem \ref{thm:x9b}) which confirms an unpublished result of Zywina \cite{Zywina}. 
\begin{theorem}\label{thm:x9}
The identity component of the Sato-Tate group of the Jacobian of the hyperelliptic curve $y^2=x^{9}+c$ is $\U(1) \times \U(1) \times \U(1).$
\end{theorem}

\begin{remark}
Shioda studies the Hodge group of curves of the form $y^2=x^m-1$  in \cite[Sections 5, 6]{Shioda82}.  In particular, Shioda shows that the Jacobian of the curve $y^2=x^9-1$ satisfies the Hodge conjecture and is a 4-dimensional abelian variety \cite[Example 6.1]{Shioda82}.  Indeed, he remarks that the Jacobian  is isogenous to the product of a CM elliptic curve $E$ and a 3-dimensional absolutely simple CM abelian variety.  The elliptic curve $E$ has $\ST^0(E)\simeq\U(1)$ and the abelian variety $A$ has $\ST^0(A) \simeq \U(1) \times \U(1) \times \U(1)$.  Thus $\ST^0(A) \times \ST^0(E) \neq \ST^0( A \times E)$, even though $A$ and $E$ do not share any common factor up to $\overline{\mathbb{Q}}$-isogeny.
\end{remark}
We also use Theorem  \ref{thm:alg} to compute $\ST^0(C_1)$, $\ST^0(C_2)$, and $\ST^0(C_3)$ for genus 2 through 10, and find surprising patterns in the shapes of the identity components for these families of curves. Following these computations, we form several conjectures (see Section \ref{sec:highergenusex}).

The remainder of this paper is organized as follows. In Section \ref{sec:background} we give some necessary background information that will be used throughout the paper.   In Section \ref{sec:genus45proof} we prove Theorems \ref{genus4} and \ref{genus5}, and in Section \ref{sec:proofsplttheorem} we prove Theorem \ref{splittheorem}. In Section \ref{sec:breakdownjacobian} we work to break down the Jacobians that appear in Theorem \ref{splittheorem} so that we can potentially use known results for the Sato-Tate groups of lower genus curves to determine the identity components of the Sato-Tate groups of higher genus curves. In Section \ref{sec:altmethod} we discuss an algorithm for computing the identity components of the Sato-Tate group.  In Section  \ref{sec:workedexamples}, we prove Theorem \ref{thm:x9} and provide an alternate proof of Theorem \ref{genus4} using this method. This algorithm requires an explicit formula for the number of points on the curve over $\mathbb F_p$ in terms of Jacobi sums, which we prove in Appendices \ref{sec:pointcount} and  \ref{sec:pointcount2}.

\section{Background}
\label{sec:background}

For the Jacobian of  a genus $g$ curve, the Sato-Tate group will be a compact subgroup of   $\USp(2g)$, which is the group of $2g\times 2g$ complex unitary matrices preserving a fixed symplectic form. In what follows, we describe the possible forms of the identity components of the Sato-Tate groups.

Let $u\in \U(1) :=\{e^{i\theta}: \theta \in [0,2\pi)\}$.  We then  define the following subgroups of $\USp(2n)$
$$\U(1)_n:=\left\langle \diag(u, \overline u, \ldots, u,\overline u): u\in \U(1)\right\rangle$$
and 
$$\U(1)^{n}:=\left\langle \diag( u_1,\overline{u_1},\ldots, u_{n},\overline{u_{n}}): u_i\in \U(1)\right\rangle.$$

As we will see in later sections, the identity components of the Sato-Tate groups we study will be products of these groups.\\

We use the following theorem of Kani and Rosen, specified to suit our problem, to express the Jacobian of a curve $C$ into the product of Jacobians of curves of smaller genus. 

\begin{theorem}\cite[Theorem C]{KaniRosen1989}\label{theorem:kanirosen}
Let $k$ be a positive integer. Let $C$ be a curve of genus $g$ and let $\alpha_i$ be an element of the automorphism group of $C$, for $i =1,\dots, k$. Suppose that
\begin{enumerate}
\item $\langle\alpha_i\rangle \cdot \langle\alpha_j\rangle = \langle\alpha_j\rangle \cdot \langle\alpha_i\rangle$, for $i,j =1,\dots,k$;
\item $g=  g_1 + \dots + g_k$, where $g_i$ is the genus of the curve $C/\langle \alpha_i \rangle$, for $i =1,\dots,k$; and
\item the genus of the curve $C/\langle \alpha_i, \alpha_j \rangle$ is 0 for all $1 \leq i \not = j \leq k$.
\end{enumerate}
Then, we have the $\overline{\mathbb Q}$-isogeny
$$\Jac(C)\sim \Jac(C/\langle \alpha_1 \rangle)\times \cdots \times \Jac(C/\langle \alpha_k \rangle).$$
\end{theorem}

\subsection{Gauss and Jacobi Sums}\label{sec:gaussjacobi}

Let  $p$ be a prime and $\mathbb{F}_q$ be a finite field with $q=p^f$ elements. We define the standard trace map $\Tr:\mathbb F_q\rightarrow \mathbb F_p$ by 
$$\Tr(x)=x+x^p+\cdots + x^{p^{f-1}}.$$
Let  $\zeta_p=e^{2\pi i/p}$ be a $p^\text{th}$ root of unity. Then for $\chi \in\widehat{\mathbb F_q^{\times}}$ we define the Gauss sum $g(\chi)$ to be
\begin{equation}\label{eqn:gausssum}
    g(\chi):=\sum_{x\in\mathbb F_q} \chi(x)\zeta_p^{\Tr(x)},
\end{equation}
where we extend $\chi$ to all of $\mathbb F_q$ by setting $\chi(0)=0$ (see, for example, \cite[Chapter 8]{IR90}). Note that $g(\epsilon)=-1$, where $\epsilon$ is the trivial character. If $\chi$ is nontrivial and if $\overline\chi$ denotes its inverse, then $g(\chi)g(\overline\chi)=\chi(-1)q$.

Let $\theta\colon  \mathbb F_p\rightarrow \mathbb C$ be the additive character defined by $\theta(x)=\zeta_p^x$, so that $g(\chi):=\sum_{x\in\mathbb F_p} \chi(x)\theta(x)$. We will make use of the following identity from \cite{Fuselier10}.

\begin{lemma}\cite[Lemma 2.2]{Fuselier10}\label{lemma:charsum}
Let $\alpha\in\mathbb F_p^\times$. Then
$$\theta(\alpha)=\frac{1}{p-1}\sum_{i=0}^{p-2}G_{-i}T^i(\alpha),$$
where $T$ is a fixed generator for the character group and $G_{-i}$ is the Gauss sum $g(T^{-i})$.
\end{lemma}

For two multiplicative characters $A,B$ over $\mathbb F_p$, we define their Jacobi sum by
$$J(A,B)=\sum_{x\in \mathbb F_q} A(x)B(1-x).$$
We have the following connection between Gauss sums and Jacobi sums (see, for example, \cite[Chapter 2]{BerndtEvansWilliams}). For nontrivial characters $A$ and $B$ over $\mathbb F_q$ whose product is also nontrivial, we have 
\begin{equation}\label{eqn:jacobigauss}J(A,B)=\frac{g(A)g(B)}{g(AB)}.\end{equation}
On the other hand, if $\phi$ is a quadratic character then $J(\phi,\phi)=-\phi(-1)$.

\section{Proofs of Theorems \ref{genus4} and \ref{genus5}}\label{sec:genus45proof}

\subsection{The Curve $y^2=x^{10} +c$.} 
\begin{theorem}\label{genus4b}
The identity component of the Sato-Tate group of the Jacobian of the hyperelliptic curve $y^2=x^{10}+c$ is $\U(1)_2 \times \U(1)_2.$
\end{theorem}
\begin{proof}  Consider the genus $g=4$ curve $C\colon y^2=x^{10} +c$. We decompose the Jacobian of our curve $C$ via suitable automorphisms  in such a way to apply Theorem \ref{theorem:kanirosen} effectively.  We let $\alpha, \beta\colon  C \to C$ be the following automorphisms of $C$
$$\alpha(x,y) = \left(c^{1/5}x^{-1}, c^{1/2}\frac{y}{x^{5}}\right),$$
and
$$\beta(x,y) = \left(c^{1/5}x^{-1}, -c^{1/2}\frac{y}{x^{5}}\right).$$
We verify the conditions of Theorem \ref{theorem:kanirosen} for $\alpha$ and $\beta$. We first find that 
\begin{equation}\label{commeq}
\alpha \beta(x,y) = \beta\alpha(x,y) = (x,-y).
\end{equation}
Via the Hurwitz genus formula, one has $g_\alpha = g_\beta = g/2$, where $g_\alpha$ and $g_\beta$ are the genuses of the curves $C/\langle\alpha\rangle$ and $C/\langle\beta\rangle$ respectively. One similarly verifies that $g_{\alpha, \beta} = 0$, where $g_{\alpha, \beta}$ denotes the genus of the curve $C/\langle\alpha, \beta\rangle$, so that all the conditions of Theorem \ref{theorem:kanirosen} are verified. We thus have the $\overline{\mathbb Q}$-isogeny
\begin{equation}\label{splitg4}
\Jac(C)\sim \Jac(C/\langle \alpha \rangle)\times \Jac(C/\langle \beta \rangle) \sim \Jac(C/\langle \alpha \rangle)^2,
\end{equation}
where the second isogeny holds via the isomorphism $C/\langle \alpha \rangle \to C/\langle \beta \rangle$ given by $(x,y) \mapsto (-x,y)$.  Thus, $\Jac(C)$ is isogenous to the square of an abelian variety. 

Now note that $\phi\colon (x,y)\mapsto(x^2,y)$ is a map from $C$ to the curve $C'\colon  y^2 = x^{5}+c$, so that we have
$$\Jac(C)\sim \Jac(C')\times A$$
for some abelian variety $A$ of dimension $2$.  By Equation \eqref{splitg4} we know that $\Jac(C)$ is isogenous to the square of an abelian variety. Since $\End(\Jac(C'_{\overline{\mathbb Q}}))_\mathbb Q \simeq\mathbb Q(\zeta_5)$ we have that $\Jac(C')$ is simple and we must therefore have that 
$$\Jac(C)\sim \Jac(C')^2.$$
It is shown in \cite{Fite2012} that  the identity component of the Sato-Tate group of $\Jac(C')$ is
$$\ST^0(C')= \U(1)\times \U(1),$$
which in turn concludes the proof that 
$$\ST^0(C) = \U(1)_2 \times \U(1)_2.$$
\end{proof}

\subsection{The Curve $C\colon  y^2 = x^{12}+c$.} 
\begin{theorem}\label{thm:genus5b}
The identity component of the Sato-Tate group of the Jacobian of the hyperelliptic curve $y^2=x^{12}+c$ is $\U(1)_2 \times \U(1)_3.$
\end{theorem}
\begin{proof}
Consider the genus $g=5$ curve $C\colon  y^2=x^{2g+2} +c$. As in the proof of Theorem \ref{genus4b}, we let $\alpha, \beta\colon  C \to C$ be the following automorphisms of $C$
$$\alpha(x,y) = \left(c^{1/6}x^{-1}, c^{1/2}\frac{y}{x^{6}}\right),$$
and
$$\beta(x,y) = \left(c^{1/6}x^{-1}, -c^{1/2}\frac{y}{x^{6}}\right).$$
However, in order to apply Theorem \ref{theorem:kanirosen} effectively, we require an additional automorphism of $C$. Namely, we let $\gamma\colon  C \to C$ be defined by
$$\gamma(x,y) = (\zeta_{3}x,y),$$
where $\zeta_{3}$ is a primitive 3rd root of unity. We may now check the conditions of Theorem \ref{theorem:kanirosen} for the automorphisms $\alpha, \beta$ and $\gamma$.  We first find that 
\begin{equation}\label{commeq2}
\alpha \beta(x,y) = \beta\alpha(x,y) = (x,-y).
\end{equation} We readily check that 
$$\langle \alpha \rangle \cdot \langle\gamma\rangle = \langle\gamma\rangle \cdot \langle \alpha \rangle,$$
and
$$\langle \beta \rangle \cdot \langle\gamma\rangle = \langle\gamma\rangle \cdot \langle \beta \rangle,$$
so that with  Equation \eqref{commeq2} the first condition of  Theorem \ref{theorem:kanirosen} holds. Now by the Hurwitz genus formula, we find that $g_\alpha = g_\beta =  \frac{g-1}{2}$, and that $g_\gamma = 1$, so that the second condition holds. Finally the third condition holds as $\alpha\beta$ is the hyperelliptic map. We thus have the isogeny
\begin{equation}\label{splitg5}
\Jac(C)\sim \Jac(C/\langle \alpha \rangle)\times \Jac(C/\langle \beta \rangle) \times \Jac(C/\langle \gamma \rangle) \sim \Jac(C/\langle \alpha \rangle)^2 \times E_1,
\end{equation}
where $E_1$ is the elliptic curve defined by $E_1\colon  y^2 = x^4 + c$. Now let $E_2\colon  y^2=x^3+c$ be an elliptic curve. Note that there exist two maps, $\phi_1\colon  C \to E_1$ and $\phi_2\colon C\to E_2$, where the maps are given by $\phi_1(x,y)= (x^3,y)$ and $\phi_2(x,y)= (x^4,y)$.

Let $\zeta_{12}$ be a primitive $12^{th}$ root of unity, and let $a=\zeta_{12}\sqrt[12]{c}$. The change of variables $x \mapsto ax$ and $y\mapsto a^6y$ transforms $C$ to the model $C'\colon    y^2=x^{12}+1$. Computing with Magma \cite{Magma}, we find $C'/\langle\alpha\rangle$ to be the genus 2 curve given by
$$C'/\langle\alpha\rangle \colon  y^2 = x^6 - 6x^4 + 9x^2 -2.$$
We have a map $\phi_3 \colon   C'/\langle\alpha\rangle \to E_3$, where $\phi(x,y)=(x^2,y)$ and  $E_3\colon  y^2 = x^3 - 6x^2 + 9x -2$, which is an elliptic curve that has CM by $\mathbb{Q}(i)$. Hence, via the maps $\phi_2$ and $\phi_3$, we have that 
$$\Jac(C'/\langle\alpha\rangle) \sim E_2 \times E_3 \sim E_2 \times E_1,$$
where the second isogeny holds since, up to $\overline{\mathbb Q}$-isogeny, there is only one elliptic curve with CM by orders in $\mathbb Q(i)$. Hence,
$$\Jac(C) \sim E_2^2 \times E_1^3.$$
We thus conclude that
$$\ST^0(C) = \U(1)_2 \times \U(1)_3.$$

\end{proof}

\section{Splitting of the Jacobians }\label{sec:proofsplttheorem}

We will first prove two lemmas that give a partial splitting of the Jacobian of the curve $C\colon   y^2= x^{2g+2}+c$ in the case that $g$ is even or odd. We will build from these two cases to give a proof of Theorem \ref{splittheorem}.

\begin{lemma}\label{evencase}
Let $g=2k$ an even integer, and $C\colon y^2= x^{2g+2}+c$. Then
$$\Jac(C) \sim \Jac(\widetilde{C})^2,$$
where $\widetilde{C}\colon  y^2= x^{g+1}+c$.
\end{lemma}

\begin{proof}
We have a map, $\phi\colon  C \rightarrow \widetilde{C}$, given by $\phi(x,y)= (x^2,y)$. Moreover, we have an automorphism, $\alpha$, of $C$ given by $\alpha(x,y)= (c^{\frac{1}{g+1}}x^{-1}, c^{1/2}yx^{-(g+1)})$. This in turn induces a second map $\widetilde{\phi}\colon  C \rightarrow \widetilde{C}$ via 
$$\widetilde{\phi}(x,y)= \phi(\alpha(x,y))= \phi(c^{\frac{1}{g+1}}x^{-1}, c^{1/2}yx^{-(g+1)})= (c^{\frac{2}{g+1}}x^{-2}, c^{1/2}yx^{-(g+1)}).$$ 
As noted in \cite[Section 5.2]{Fite2016}, in order to prove the Lemma it is sufficient to check that the pullbacks of a basis of differential forms for $\Jac(\widetilde{C}$) via $\phi$ and $\widetilde{\phi}$ give a basis for the space of differential forms for $\Jac(C$). A basis for the space of regular $1$-forms of the Jacobian of a hyperelliptic curve of genus $g$ is given by forms $x^idx/y$ for $i=0,\cdots, g-1$ (see, for example, \cite[Section 3]{VanWamelen1998}). We thus compute
$$\phi^*\left(x^i\frac{dx}{y}\right) = \frac{x^{2i}d(x^2)}{y} = 2 \frac{x^{2i+1} dx}{y},$$
and 
$$\widetilde{\phi}^*\left( x^i\frac{dx}{y}\right) = \frac{c^{\frac{2i}{g+1}-\frac{1}{2}}x^{-2i}d\left(\frac{c^{\frac{2}{g+1}}}{x^2}\right)x^{g+1}}{y}= -2c^{\frac{2(i+1)}{g+1} - \frac{1}{2}} \frac{x^{g-2-2i}dx}{y}.$$ 
The only thing that remains to be checked is that 
$$\{x^{2j+1}, x^{g-2-2j}\; |\; j=0, \cdots, \frac{g}{2}-1\} = \{x^i\; |\; i=0,\cdots, g-1\}.$$
However, to obtain even exponents, say $x^{2m}$, in the set of the left hand side of the equation,  we may take $j= g/2 - (m+1)$ with $x^{g-2-2(g/2-m -1)}= x^{2m}$. For all of the odd exponents, say $x^{2m+1}$, we may take $j= m$ with $x^{2j+1}$.  
\end{proof}

\begin{lemma}\label{oddcase}
Let $g=2k+1$ be an odd integer, and $C\colon  y^2=x^{2g+2}+c.$ Then 
$$\Jac(C) \sim \Jac(\widetilde{C}) \times \Jac(C'),$$
where $\widetilde{C} \colon  y^2 = x^{g+1} + c$ and $C'\colon  y^2 = x^{g+2} + cx$ are curves of genus $k$ and $k+1$ respectively.
\end{lemma}

\begin{proof}
We have a map $\phi\colon  C \rightarrow \widetilde{C}$, given by $\phi(x,y)=(x^2,y)$, and a map $\widetilde{\phi}\colon  C \rightarrow C'$, given by $\widetilde{\phi}(x,y) = (x^2,xy)$. We now only need to check that the pullbacks of the basis elements for the space of regular $1$-forms of the Jacobians of $\widetilde{C}$ and $C'$ give a basis for the space of regular $1$-forms of $\Jac(C)$. We therefore compute
$$ \phi^*\left(x^i \frac{dx}{y}\right) = 2\frac{x^{2i+1}dx}{y},$$
while 
$$\widetilde{\phi}^*\left(x^i \frac{dx}{y}\right) = \frac{x^{2i}d(x^2)}{xy} = 2 \frac{x^{2i} dx}{y}.$$
Now, in the first case, as $i$ runs through $0, \cdots, k-1,$ we get all the odd forms corresponding to $x, \cdots , x^{2k-1}.$ In the second case we get all of the even ones, and this concludes the proof.
\end{proof}

We are now in a position to prove the following theorem.
\begin{theorem}\label{splitting}
Let $v_2\colon  \mathbb Q^* \rightarrow \mathbb Z$ denote the $2$-adic valuation map, i.e. $v_2(a/b) = \alpha$, where $\frac{a}{b}=2^\alpha\frac{e}{d}$ and $p$ does not divide $e$ or $d$. Let $C_1\colon y^2= x^{2g+2}+c$ be a hyperelliptic curve of genus $g$ and write $k:= v_2(g+1)$. Then we have the following isogeny over $\overline{\mathbb Q}$
$$\Jac(C_1) \sim \Jac(y^2= x^{{(g+1)}/{2^k}}+c)^2 \times \prod_{i=0}^{k-1} \Jac(y^2=x^{{(g+1)}/{2^{i}}+1}+cx).$$
\end{theorem}

\begin{proof}
Let $k:= v_2(g+1)$. We will prove the result by induction on $k$. If $k=0$, then $g$ is even and we have already shown that
$$\hbox{Jac}(C_1) \sim \hbox{Jac}(y^2= x^{g+1}+c)^2.$$
If $k=1$, then $g= 2a-1$ (with $(2,a)=1$), and by our result for odd genus, we have
$$\hbox{Jac}(C_1) \sim \hbox{Jac}(y^2= x^{g+1}+c) \times \hbox{Jac}(y^2 = x^{g+2}+cx).$$
Now, $g+1= 2a = 2(2b+1)= 2 (2b) +2$, for some integer $b$, and our result for even genus case implies that
$$\hbox{Jac}(C_1) \sim \hbox{Jac}(y^2= x^{{(g+1)}/{2}}+c)^2 \times \hbox{Jac}(y^2=x^{g+2}+cx).$$
By induction, we suppose that our result holds for $l$ and suppose $v_2(g+1)=l+1$. Then by our result for odd genus, we have
$$\hbox{Jac}(C_1) \sim \hbox{Jac}(y^2=x^{g+1}+c)\times \hbox{Jac}(y^2=x^{g+2}+cx).$$
By assumption, $g+1= 2^{l+1}d$ (with $d=2e +1$, for some integer $e$), so that $g+1= 2^{l+1}(2e+1) = 2g' +2$, where $g'=2^{l+1}e+2^l-1$. Thus $v_2(g'+1)= l$, and we may therefore use our induction hypothesis to conclude that
\begin{align*}
\hbox{Jac}(C_1)& \sim \left(\hbox{Jac}(y^2= x^{{(g'+1)}/{2^l}}+c)^2 \times \prod_{i=0}^{l-1} \hbox{Jac}(y^2=x^{{(g'+1)}/{2^{i}}+1}+cx)\right)\\&\hspace{1in}\times \hbox{Jac}(y^2=x^{g+2}+cx)\\
& \sim \hbox{Jac}(y^2= x^{{(g+1)}/{2^{l+1}}} +c)^2 \times \prod_{i=0}^{l} \hbox{Jac}(y^2=x^{{(g+1)}/{2^{i}}+1}+cx),
\end{align*}
since $g'=(g-1)/2$. 
\end{proof}

\section{A Further Splitting of The Jacobians of Theorem \ref{splittheorem}}\label{sec:breakdownjacobian}

Note that the curve $y^2=x^{{(g+1)}/{2^{k-1}}+1}+cx$ that appears in Theorem \ref{splittheorem} has odd genus since 
$$\frac{g+1}{2^{k-1}}+1 =2\left(\frac{g+1}{2^{k}}\right)+1 , $$
and $v_2(g+1)=k$ implies that $\frac{g+1}{2^{k}}$ is odd. In this section we show how to further split curves of this form.

Let $g$ be an odd integer and $C\colon y^2= x^{2g+1}+cx$ be a genus $g$ curve. Let $E\colon y^2=x^3+cx$ be an elliptic curve. Throughout this section, we work over the field $\mathbb F=\mathbb Q(\zeta, c^{1/g})$, where $\zeta=\zeta_g$ is a primitive $g$-th root of unity. The morphism $\phi\colon  C \rightarrow E$ defined by
$$\phi(x,y) = \left(x^{g}, yx^{(g-1)/2}\right)$$
is a non constant morphism from $C$ to the elliptic curve $E$. We would like to find  more morphisms from $C$ to families of lower genus curves.

Our ultimate goal is to be able to further break down our result from Theorem  \ref{splittheorem}  so that we may write the Jacobians of curves as a product of Jacobians of lower genus curves. Ideally, we would like to be able to write the Jacobian as a product of Jacobians of elliptic curves (genus 1) or genus 2 curves since the Sato-Tate groups of these lower dimension Jacobians are completely classified (see \cite{Barnet2011,Fite2012, Harris2010}).

\subsection{Morphisms to Lower Genus Curves}
For $i=0,1$ we define the curve $C_i$ to be 
$$C_i\colon y^2=\sum^{(g-1)/2}_{k=0} (-1)^{k}\left[\binom{g-k}{k} + \binom{g-k-1}{k-1}\right]\zeta^{ik}c^{k/g} x^{g-2k}.$$
Note that this is a curve of genus $g'=(g-1)/2$ and it is defined over $\mathbb F$. The following table gives $C_i$ for small values of $g$ and for $c=1$.  \\

\begin{tabular}{c|l}
    genus  of $C$& curve $C_i$\\
    \hline 
    3 & $y^2=x^3-3\zeta^{i}x$\\
    5 & $y^2=x^5-5\zeta^{i}x^3+5\zeta^{2i}x$\\
    7 & $y^2 = x^7-7\zeta^{i} x^5+14\zeta^{2i}x^3-7\zeta^{3i}x$\\
    9 & $y^2 = x^9 -9\zeta^{i} x^7+27\zeta^{2i}x^5 - 30\zeta^{3i}x^3 +9 \zeta^{5i} x$\\
    11 & $y^2=x^{11}-11\zeta^{i}x^9+44\zeta^{2i}x^7-77\zeta^{3i}x^5+55\zeta^{4i}x^3-11\zeta^{5i}x$
\end{tabular}

\begin{lemma}
The map 
$$\phi_i(x,y) = \left(\dfrac{x^2+\zeta^{i}c^{1/g}}{x}, \dfrac{y}{x^{a}} \right),$$
where $a=\frac{g+1}{2}$, is a nonconstanst morphism from $C$ to $C_i$.
\end{lemma}

\begin{proof}
The proof relies on the following identity attributed to E.H. Lockwood (see, for example, \cite[Section 9.8]{Koshy2014}).
$$A^n+B^n = \sum^{\lfloor n/2\rfloor}_{k=0} (-1)^{k}\left[\binom{n-k}{k} + \binom{n-k-1}{k-1}\right](AB)^k(A+B)^{n-2k},$$
where $n\geq1$ and $\binom{r}{-1}=0$. 
Letting $n=g$, $A=x^2$, and $B=\zeta^ic^{1/g}$ yields
$$x^{2g}+c = \sum^{\frac{g-1}{2}}_{k=0} (-1)^{k}\left[\binom{g-k}{k} + \binom{g-k-1}{k-1}\right]\zeta^{ik}c^{k/g}x^{2k}(x^2+\zeta^ic^{1/g})^{g-2k},$$
since $\zeta^{ig}=1$. We multiply both sides by $x$ to get
\begin{align}\label{eqn:lockwood}
    x^{2g+1}+cx = \sum^{\frac{g-1}{2}}_{k=0} (-1)^{k}\left[\binom{g-k}{k} + \binom{g-k-1}{k-1}\right] \zeta^{ik}c^{k/g}x^{2k+1}(x^2+\zeta^ic^{1/g})^{g-2k}.
\end{align}

We now demonstrate that $\phi_i$ is indeed a morphism between $C$ and $C_i$. We apply the transformation of variables to $C_i$ to get
\begin{align*}
    \left(\frac{y}{x^a}\right)^2 &=\sum^{(g-1)/2}_{k=0} (-1)^{k}\left[\binom{g-k}{k} + \binom{g-k-1}{k-1}\right]\zeta^{ik} c^{k/g}\left(\dfrac{x^{2}+\zeta^{i}c^{1/g}}{x}\right)^{g-2k}\\
    y^2&=\sum^{(g-1)/2}_{k=0} (-1)^{k}\left[\binom{g-k}{k} + \binom{g-k-1}{k-1}\right]\zeta^{ik}c^{k/g}x^{2k+1} \left(x^{2}+\zeta^{i}c^{1/g}\right)^{g-2k}\\
    &=x^{2g+1}+cx,
\end{align*}
where the last equality holds by Equation \eqref{eqn:lockwood}. Hence, we have shown that $\phi_i$ is a morphism from $C$ to $C_i$.

\end{proof}

\subsection{Pullback of Differentials}
We claim that 
$$\text{Jac}(C) \sim E \times A,$$
where $\sim$ denotes isogeny over $\overline{\mathbb Q}$ and  $A$ is an abelian variety defined over $\mathbb Q$ for which $A\sim \text{Jac}(C_0) \times \text{Jac}(C_1)$.
As noted in \cite[Section 5.2]{Fite2016}, in order to prove this claim it is sufficient to check that there is an isomorphism of $\mathbb F$-vector spaces of regular differential forms  
$$\Omega_{C} =\phi^*(\Omega_{E}) \oplus \phi^*_0(\Omega_{C_0})\oplus  \phi^*_{1}(\Omega_{C_1}).$$

As noted in Section \ref{sec:proofsplttheorem}, a basis for the space of regular
$1$-forms of the hyperelliptic curve $C$ of genus $g$ is given by the forms $\omega_j=x^jdx/y$ for $j=0,\ldots, g-1$ . Similarly, for  both of the curves $C_i$, we have the following basis
$$\left\{\frac{dx}{y}, \frac{xdx}{y}, \ldots, \frac{x^{\frac{g-1}{2}-1}dx}{y} \right\}$$
since they are both hyperelliptic curves of genus $\frac{g-1}{2}$.
For the elliptic curve, we will use the nowhere vanishing differential $dx/y$.\\

We let $a=(g+1)/2$ and first note that 
\begin{equation}\label{differentialE}
\phi^*\left(\frac{dx}{y}\right) = \frac{d(x^{g})}{yx^{(g-1)/2}} = \frac{gx^{(g-1)/2}dx}{y} = gx^{a-1}\omega_0. 
\end{equation}
Furthermore, let $m$ be some integer between 0 and $\frac{g-1}{2}-1$. Then
\begin{align*}
    \phi^*_i\left(\frac{x^mdx}{y}\right) &= \frac{\left(\frac{x^2+\zeta^{i}c^{1/g}}{x}\right)^md\left(\frac{x^2+\zeta^{i}c^{1/g}}{x}\right)}{yx^{-a}}\\
    &=\frac{\left(\sum_{k=0}^{m}\binom{m}{k}x^{m-2k+a}\zeta^{ik}c^{k/g} \right)dx}{y}-\zeta^{i}c^{1/g}\frac{\left(\sum_{k=0}^{m}\binom{m}{k}x^{m-2k-2+a}\zeta^{ik}c^{k/g} \right)dx}{y}\\
    &= f_{i,m}(x) \frac{dx}{y},\numberthis \label{eqn:differential}
\end{align*}
where $f_{i,m}$ are polynomials given by
\begin{equation*}
f_{i,0}(x) = x^a - \zeta^i c^{1/g} x^{a-2},
\end{equation*}
and
\begin{equation}\label{polybasism}
f_{i,m}(x) = x^{m+a} + \sum_{k=1}^{m}\left(\binom{m}{k} - \binom{m}{k-1}\right) x^{m-2k+a}\zeta^{ik}c^{k/g} + \zeta^{i(m+1)}c^{\frac{m+1}{g}} x^{a-m-2}, 
\end{equation}
if $m>0$.
\begin{claim}
Given an integer $0 \leq n \leq \frac{g-1}{2} -1$, the set of polynomials $P_n := \{f_{i,m} | i=0,1; 0 \leq m \leq n\} \cup \{x^{a-1}\} $ forms a linearly independent set.
\end{claim}
\begin{proof}
We argue by induction on $n$. We note that $f_{0,0}$ is of degree $a$, while $x^{a-1}$ is of degree $a-1$ and
$$f_{0,0}(x) - f_{1,0}(x) = (1-\zeta c^{1/g}) x^{a-2},$$
a polynomial of degree $a-2$, so that the claim holds for $n=0$. For $n \geq 1$, we let $\{\lambda_{i,k}\}_{i=0, 1; 0\leq k \leq n}$ be scalars such that 
\begin{equation}\label{LI} 
\sum_{i=0, 1; 0\leq k \leq n} \lambda_{i,k} f_{i,k} + \lambda_{a-1} x^{a-1} = 0.
\end{equation}
We note that $f_{0,n}$ and $f_{1,n}$ are the only two polynomials in our family that are of degree $a + n$, so that (\ref{LI}) holds only if
\begin{equation}\label{topdegreeLI}
\lambda_{0,n} + \lambda_{1,n} = 0,
\end{equation}
by looking at the leading coefficients of $f_{i,n}$ in (\ref{polybasism}). Moreover, $f_{0,n}$ and $f_{1,n}$ are the only two polynomials in our family that contain a monomial of degree $a -n -2$. We, therefore, must have that 
\begin{equation}\label{bottomdegreeLI}
\lambda_{0,n} + \lambda_{1,n} \zeta^{n+1} = 0.
\end{equation}
Since $n < g$ and $\zeta$ is a primitive $g$-th root of unity, $\zeta^{n+1} \not= 1$, and together with equations (\ref{topdegreeLI}) and (\ref{bottomdegreeLI}) this implies that $\lambda_{0,n} = \lambda_{1,n} = 0$. The set of remaining polynomials in the family is now $P_{n-1}$ and, by induction, this implies that the remaining $\lambda_{i,k} = 0$ for all $i=0,1$ and $0\leq k \leq n - 1$, and $\lambda_{a-1} =0$, proving our claim.
\end{proof}
By the above claim for $n= \frac{g-1}{2} - 1$, the family $P_n$ exhibits $g$ linearly independent polynomials inside the $g$-dimensional vector space of polynomials of degree less than or equal to $g-1$. In particular $P_n$ is a basis for that space. We can thus write a basis for $\phi^*(\Omega_{E_F}) \oplus \phi^*_0(\Omega_{C_0})\oplus  \phi^*_{1}(\Omega_{C_1})$ that is also a basis for $\Omega_{C}$, via (\ref{differentialE}) and (\ref{eqn:differential}). Thus, we have proved the following. 
\begin{proposition}
$$\Jac(C) \sim E \times \Jac(C_0) \times \Jac(C_1),$$
where $\sim$ denotes isogeny over $\overline{\mathbb Q}$.
\end{proposition}

\section{A New Algorithm to Compute $\ST^{0}(C)$}\label{sec:altmethod}

In this section we describe an algorithm to compute the identity component of the Sato-Tate group of the Jacobian for curves of the form 
$$C_1\colon y^2=x^{2g+2} +c, \quad  C_2\colon y^2=x^{2g+1} +c, \quad C_3\colon y^2=x^{2g+1} +cx,$$ 
where $c\in\mathbb Q^*$ is a constant. Note that the Jacobians of the curves in all three families are CM abelian varieties (see, for example, \cite{Muller2017} or \cite{Wolfart2001}).  We show that the algorithm coincides with our result for the curve $y^2=x^{10}+c$. We then use this method to prove that the identity component of the Sato-Tate group of $y^2=x^9 +c$ is $\U(1) \times \U(1) \times \U(1)$, which confirms an unpublished result of Zywina \cite{Zywina}.  We then compute $\ST^{0}(C_1),\ST^0(C_2),$ and $\ST^0(C_3)$ for genus 2 through 10 which give evidence for several conjectures which we formulate.

\subsection{Preliminaries}\label{sec:altmethod1}
We begin by defining the Sato-Tate group $\ST(A)$ of an abelian variety $A/K$, where $K$ is a number field, of dimension $g$ as in \cite[Section 3.2]{SutherlandAWSNotes} and \cite[Chapter 15]{Mur93}.

For an odd prime $\ell$, the Tate module is defined as $T_{\ell}:=\varprojlim_{n} A[\ell^n]$ to be a free $\mathbb{Z}_{\ell}$-module of rank $2g$, and the rational Tate module is defined as $V_{\ell}:=T_{\ell}\otimes_{\mathbb{Z}} \mathbb{Q}$ to be a $\mathbb{Q}_{\ell}$-vector space of dimension $2g.$ The Galois action on the Tate module is given by an $\ell$-adic representation $$\rho_{A,\ell}:\Gal(\overline{K}/K) \rightarrow \Aut(V_{\ell}) \cong \GL_{2g}(\mathbb{Q}_{\ell}).$$ 
Let $G_{\ell}$ denote the image of this map. We let $G^{\Zar}_{\ell}$ denote the Zariski closure of $G_{\ell}$ in $\GL_{2g,\mathbb{Q}_{\ell}}$ (as an algebraic group), and we define $G^{1,\Zar}_{\ell}$ by adding the symplectic constraint  $M^t\Omega M=\Omega$, where
$$\Omega:=\begin{pmatrix}~&-I_g\\I_g&~ \end{pmatrix},$$
so that $G_{\ell}^{1,\Zar}$ is a subgroup of $\Sp_{2g,\mathbb{Q}_{\ell}}$.

Choose an embedding $\iota\colon \mathbb{Q}_{\ell} \rightarrow \mathbb{C}$ and use it to define $G_{\ell,\iota}^{1,\Zar}(\mathbb{C})$, which is unique up to conjugacy. We then define $\ST(A) \subseteq \USp(2g)$ as a maximal compact subgroup of $G_{\ell,\iota}^{1,\Zar} (\mathbb{C})$ (unique up to conjugacy).  

Over an appropriate cyclotomic field $k$, the Tate module of the Jacobian splits into a sum of one-dimensional Galois characters (see, for example, \cite[Example 1.2]{Serreladicrep}).   This allows us to apply some results from group theory.  The $\ell$-adic monodromy group $G_\ell^{\Zar}$ is equal to the  dual of the Tate module (see \cite[Section 0]{Pink}) and so $G_\ell^{\Zar}$ is dual to the the group generated by these characters. By work of Serre \cite[Section 8.3.2]{SerreNXP}, $G_\ell^{1,\Zar}$ is the dual of the group generated by these  characters modulo the cyclotomic character. By definition, the group $\ST(A)$ is a maximal compact subgroup of  $G_{\ell,\iota}^{1,\Zar}(\mathbb{C})$,  so $\ST^{0}(A)$ is dual to the maximal torsion-free quotient of the group generated by these characters.  If all of the one-dimensional characters come from Jacobi sums, as is the case in Sections \ref{sec:workedexamples} and \ref{sec:highergenusex}, then the $p$-adic valuation map is a map from this group to an explicit abelian group and the kernel of this map is the torsion subgroup.  

Before we define the map, we first recall Stickelberger's Congruence Theorem.

\subsection{Stickelberger's Congruence Theorem}\label{sec:stickelberger}  
The background information in this section can be found in \cite{Con} and \cite[Chapters 6 and 8]{IR90}.  Let  $p$ be a prime, $\mathbb{F}_q$ be a finite field with $q=p^f$ elements, and $\zeta_p , \zeta_{q-1} \in \mathbb{C}$ be fixed roots of unity with respective orders $p$ and $q-1$. We then have the following diagram of number fields and primes 
\begin{center}
\begin{tikzcd}
\mathbb{Q}(\zeta_{q-1},\zeta_p)  \arrow{d}& \mathfrak{B}_1^{p-1} \cdots \mathfrak{B}_g^{p-1} \arrow{d} \\ 
\mathbb{Q}(\zeta_{q-1}) \arrow{d} & \mathfrak{p}_1 \cdots \mathfrak{p}_g \arrow{d}\\
\mathbb{Q} & p
\end{tikzcd}
\end{center}
where $\mathfrak{B_i}$ lies over the prime $\mathfrak{p_i}$ and  $g=\phi(q-1)/f$ (and $\phi$ is Euler's totient function). Fix any prime $\mathfrak{p}$ in $\mathbb{Q}(\zeta_{q-1})$ lying over $p$ and let $\mathfrak{B}$ be the unique prime in $\mathbb{Q}(\zeta_{q-1},\zeta_p)$ lying over $\mathfrak{p}$.  Let $\omega_{\mathfrak{p}}$ be the Teichm{\"u}ller character on $\mathbb{F}_q$.

For $0 \leq b < q-1$, write the base $p$ expansion of $b$ as 
$$b=b_0 +b_1p +\cdots +b_{f-1}p^{f-1},$$
where $0 \leq b_i \leq p-1$ and not all $b_i=p-1.$
Recall from Equation \ref{eqn:gausssum} that the Gauss sum of a multiplicative character $\chi$ of $\mathbb{F}_q$ is 
$$g(\chi):= \displaystyle\sum_{ x \in \mathbb{F}_q} \chi(x)\zeta_p^{\Tr(x)}.$$
The normalized Jacobi sum of the multiplicative characters $\chi_1,\chi_2,\dots,\chi_r$ of  $\mathbb{F}_q$ is defined by 
$$J(\chi_1,\ldots,\chi_r):=(-1)^{r}\sum_{ x_1+x_2+ \cdots +x_r=1}\chi_1(x_1) \cdots \chi_r(x_r).$$

\begin{theorem}{(Stickelberger's Congruence Theorem \cite{St1890})}\label{thm:stickelberger}
$$g(\omega_\mathfrak{p}^{-b})\equiv \dfrac{(\zeta_p-1)^{b_0 + \cdots + b_{f-1}}}{b_0! \cdots b_{f-1}!}\mod \mathfrak{B}^{b_0 + \cdots +b_{f-1}+1}.$$
\end{theorem}

We will use Stickelberger's Congruence Theorem  with $q=p$ to compute the $\mathfrak{B}$-adic valuations of the Jacobi sums arising in Theorems \ref{theorem:pointcountgeneral} and \ref{theorem:pointcount2}. Given that $\mathfrak{B}$ will always divide the quantity $(\zeta_p-1)$ exactly once, we note the following immediate consequence of Stickelberger's Congruence Theorem.

 \begin{corollary}\label{cor:stickelberger}
 Let $\ord_\mathfrak{B} \colon  \mathbb{Q}(\zeta_{p-1}, \zeta_p) \rightarrow \mathbb{Z}$ denote the $\mathfrak{B}$-adic valuation map. Then, for $0\leq b\leq p-1$,
 $${\ord}_\mathfrak{B}(g(\omega_\mathfrak{p}^{-b})) = b.$$
 \end{corollary} 

In the case of  $y^2=x^9+c$ and  $p \equiv 1 \pmod 9$, Corollary \ref{coro:d=9} tells us that the Jacobi sums that arise in the point count formula are all of the form $J(\chi^m, \phi)$ for $1 \leq m \leq 8$  where $\chi = T^{(p-1)/9}, \phi= T^{(p-1)/2}$, and $T$ is any fixed generator of the character group $\widehat{\mathbb F_p^{\times}}$. In particular, given $\mathfrak{p}$ dividing $\mathfrak{B}$, we can choose $T = \omega_\mathfrak{p}^{-1}$.

\subsection{The Map}

Let $p$ be a split prime of the CM field $K$ and let $\iota_1,\dots,\iota_n$  be the embeddings  of the field of definition of the one-dimensional Galois characters  into the algebraic closure of $\mathbb{Q}_p$.  Consider the homomorphism that sends a character $\rho$ to the $n$-tuple $\left(v_p(\iota_1(\rho(\Frob_p))),\dots,v_p(\iota_n(\rho(\Frob_p)))\right)$, where $v_p$ is the $p$-adic valuation map.  Let $T$ be a fixed generator for the character group  $\widehat{\mathbb{F}_p^{\times}}$, $\chi=T^{(p-1)/d}$ for some positive integer $d$, and $\phi=T^{(p-1)/2}$ be a quadratic character. In the case where $\rho$ is a Jacobi sum character, so that $\rho(\Frob_p)=J(\chi^m,\phi),$ we use Stickelberger's Theorem to compute a matrix whose columns are the images under this homomorphism of the characters appearing in the Tate module.

There is one such embedding for each injective map from the group of characters to the unit circle because there is one embedding for each primitive root of unity (and primitive roots of unity give these maps). We form a matrix of size $n \times k$ with this information, forming one column for each of $k$ pairs of characters in $J(\chi^m,\phi)$, and one row for each of the $n$ embeddings of the group of characters into the circle. The matrix is defined so that the $(j,m)$-th entry is the $p$-adic valuation of the Jacobi sum of the $m$th character under the $j$th embedding. 

Next we formally define this matrix, call it $M$, whose columns are the images under this homomorphism of the  characters appearing in the Tate module.

\begin{definition}\label{def:matrix}
 The matrix $M$ is constructed as follows. We define a map $\phi\colon  \mathbb Z^k \rightarrow \mathbb Z^n$, where $n$ is the number of embeddings, as a composition of two maps $\phi_1$ and $\phi_2$. Given $a=(a_1, a_2,\ldots,a_k) \in \mathbb Z^k$ any $k$-tuple of integers, $\phi_1$ maps $a\mapsto \prod \chi_i^{a_i}$, where the $\chi_i$ are the one dimensional characters coming from the Tate module\footnote{In Section \ref{sec:workedexamples}, the characters $\chi_i$ are Jacobi sums of the form $J(\chi^i,\phi)$. See \ref{sec:pointcount} and \ref{sec:pointcount2} for more detailed descriptions of the Jacobi sums that appear in our examples.}. The second map $\phi_2$ takes this character product to each of $n$ embeddings $\iota_j(\prod \chi_i^{a_i})$ and then computes the $p$-adic valuation of each embedding. The composition of the maps can be expressed as a matrix $M$.
  
  To be more precise, let $$\phi_1\colon \mathbb{Z}^k \rightarrow \widehat{\Gal(\overline{K}/{K})},$$
be the map which sends $a=(a_1,a_2,\dots,a_k)$ to $\prod_{i}\chi_i^{a_i}$. The second map, $$\phi_2\colon \widehat{\Gal(\overline{K}/{K})} \rightarrow \mathbb{Z}^n,$$
combines the embedding and the $p$-adic valuation steps: it sends a character $\rho$ to the $n$-tuple $(v_p(\iota_1(\rho(\Frob_p))),v_p(\iota_2(\rho(\Frob_p))),\dots, v_p(\iota_n(\rho(\Frob_p))))$.  The composition $\phi_2 \circ \phi_1$ forms a matrix $M$ whose  $(j,m)$-th entry is $v_p(\iota_j(\rho(\Frob_p))).$  

\end{definition}

In Sections \ref{sec:workedexamples} and \ref{sec:highergenusex}, we will form this matrix for curves in the three families $C_1, C_2,$ and $C_3$. For curves in each of the three families, the number of points on the curve over the field $\mathbb F_p$ can be expressed as a sum of Jacobi sums (see \ref{sec:pointcount} and \ref{sec:pointcount2}).  Thus, as is the case for Fermat curves in \cite{Saito2004}, the $\ell$-adic representation $\rho(\Frob_p)$ is described by the Jacobi sums that appear in the point count formulas. These Jacobi sums are the eigenvalues of the $\mathbb F_p$-Frobenius endomorphism action on the $\ell$-adic Tate module (see, for example, \cite[Section 2.1]{Ahmadi2015}).

In the case where $\rho$ is a Jacobi sum character, the matrix $M$ has $(j,m)$-th entry  $v_p(\iota_j(J(\chi^m,\phi))).$ Each entry of $M$  is 1 if the angles sum to at least $2\pi$ and zero otherwise.  A method of completing the first row, for $y^2=x^9+c$ and  $p \equiv 1 \pmod 9$, is as follows. A similar argument can be made for the remaining rows, as well as for other curves. 
\begin{lemma}
For $1 \leq m \leq 4$, we have 
$${\ord}_\mathfrak{B}(J(\chi^m, \phi)) = 0,$$
while for $5\leq m \leq 8$, we have 
$${\ord}_\mathfrak{B}(J(\chi^m, \phi))={p-1}.$$
\end{lemma}
\begin{proof}
Using Equation (\ref{eqn:jacobigauss}), we see that 
\begin{align*}
{\ord}_\mathfrak{B}(J(\chi^m, \phi)) &= {\ord}_\mathfrak{B}\left(g\left(\omega_\mathfrak{p}^{-\frac{m(p-1)}{9}}\right)\right) + {\ord}_\mathfrak{B}\left(g\left(\omega_\mathfrak{p}^{-\frac{p-1}{2}}\right)\right) - {\ord}_\mathfrak{B}\left(g\left(\omega_\mathfrak{p}^{-\frac{(2m+9)(p-1)}{18}}\right)\right)\\
& = \left\{ \begin{array}{ll} \frac{m(p-1)}{9} + \frac{p-1}{2} - \frac{(2m+9)(p-1)}{18} = 0 & \hbox{if } 1\leq m \leq 4,\\ \frac{m(p-1)}{9} + \frac{p-1}{2} - \frac{(2m-9)(p-1)}{18} ={p-1} & \hbox{if } 5 \leq m \leq 8, \end{array} \right.
\end{align*}
where the second equality holds by Corollary \ref{cor:stickelberger}.
\end{proof}
This leads us to the following theorem.  
\begin{theorem} \label{theorem:matrix}
Let $M$ be the matrix in Definition \ref{def:matrix} with $(j,m)$-th entry  $v_p(\iota_j(J(\chi^m,\phi)))$. The elements in the kernel of $M$ give the relations between characters $\chi_i$ for $i=1,2,\dots, k$, where  $k\leq g$, that determine the structure of the identity component of the Sato-Tate group of the genus $g$ curves of the form 
 $$C_1\colon y^2=x^{2g+2}+c,\;  \; C_2\colon y^2=x^{2g+1}+cx, \; \; C_3\colon y^2=x^{2g+1} +c.$$
\end{theorem}

  \begin{proof}
Let $C$ be a smooth projective curve defined over $\mathbb Q$. Recall from Section \ref{sec:altmethod1} the Tate module of the Jacobian splits into a sum of one dimensional characters and  $\ST^0(C)$ is dual to the maximal torsion-free quotient of the group generated by these characters which is $G_{\ell}^{1,\Zar}$. Let $M$ be the matrix in Definition \ref{def:matrix}. 

Since the $p$-adic valuations that make up the entries of $M$ are integers, they are not torsion, so the image of $M$ is torsion-free. Recall that the matrix is constructed using a composition of maps, see Definition \ref{def:matrix}. We claim that the kernel of the second map is torsion and so the kernel of the first map is a finite index submodule of the kernel of $M$.   Indeed,  any element in the kernel has $v_p(\iota_j(J(\chi^m,\phi)))=0$ for all $p$-adic valuations.  Moreover,  all $\ell$-adic valuations are zero since $J(\chi^m,\phi))$ acts on the $\ell$-adic Galois representations as an $\ell$-adic unit.  In addition, the absolute value must be one at all infinite places since the absolute value is independent of the complex embedding by Weil's Riemann Hypothesis \cite{Del74} and the product of the absolute value over all complex embeddings vanishes by the product formula.  Hence, $J(\chi^m,\phi)$ is a root of unity; for ease of notation, we will denote it by $\chi_m$. Because this holds for all split $p$, the image of the character consists of roots of unity, so it has finite order.

We can easily determine the elements $a$ in the kernel of the first map by computing the kernel of $M$ since the kernel of the first map is a finite index submodule of the kernel of $M$. Setting $\chi_1^{a_1}\chi_2^{a_2}\cdots\chi_k^{a_k}=1$ for each element $a$ in the kernel of $M$ gives  a set of relations on the characters $\chi_1,...,\chi_k$.
 Thus, we can express the list of characters in the form
$$\{\chi_{b_1}, \overline{\chi_{b_1}}, \ldots, \chi_{b_r}, \overline{\chi_{b_r}}\},$$
where there are $t_i$ copies of each pair $\chi_{b_i}, \overline{\chi_{b_i}}$ for some positive integers $t_i$ satisfying $\sum{t_i} = g$. Note that the characters in this list may not be independent since a character may just be the product of other characters in the list. Thus, a list of independent characters will be  
$$\{\chi_{c_1}, \overline{\chi_{c_1}}, \ldots, \chi_{c_h}, \overline{\chi_{c_h}}\},$$
where there are $r_i$ copies of each pair $\chi_{h_i}, \overline{\chi_{h_i}}$ for some positive integers $r_i$ satisfying $\sum{r_i} \leq g$. 

Since the characters are roots of unity,  we will denote them by $u_j:=\chi_j$ to match the notation of Section \ref{sec:background}. We claim that we can then write
$$\ST^0(C)= \langle\diag(u_{c_1}, \overline{u_{c_1}},\ldots, u_{c_h}, \overline{u_{c_h}})\;|\; u_{c_i} \overline{u_{c_i}}=1\rangle,$$
 where there are $r_i$ copies of each pair $u_{c_i}, \overline{u_{c_i}}$ for some positive integers $r_i$ satisfying $\sum{r_i} \leq g$.  The claim follows since by construction the columns of the matrix are images under the above described homomorphism of the characters appearing in the Tate module, and we have shown the kernel of the first map is a normal finite index subgroup of the kernel of the matrix (see \cite[pg. 31]{SutherlandAWSNotes}).

\end{proof}
 
 \begin{remark}
By \cite[Definition 4.1]{SutherlandAWSNotes} or \cite[Section 8.2]{SerreNXP} each element of the form $\diag(u_{c_1}, \overline{u_{c_1}},\ldots, u_{c_h}, \overline{u_{c_h}})$ is a Hodge circle by Serre's definition and the Hodge circles generate a dense nontrivial subgroup of $\ST^0(C)$. 
 
 \end{remark}
\subsection{Algorithm to Compute $\ST^{0}(C)$}\label{sec:algorithm}

We use this theory to efficiently compute  $\ST^{0}(C)$, for the curves $C_1$, $C_2$, and $C_3$, with the following algorithm. 
\begin{algorithm}\label{algorithm}
\begin {enumerate}
\item  Use Theorems  \ref{theorem:pointcountgeneral} and \ref{theorem:pointcount2}  to determine which characters contribute to the Jacobi sums.  
\item Form the matrix $M$ of Definition \ref{def:matrix}.  For the columns use the  appropriate $J(\chi^m,\phi)$, and for the rows use the embeddings into the circle.  By the composition of the embeddings with $J(\chi^m,\phi)$ we mean take the composition of the embedding with each of the characters $\chi^m$ and $\phi$. The entries in the matrix are 1 if the sum of the angles is at least $2\pi$ and 0 otherwise.
\item Compute the kernel of the matrix $M$.
\item Note that the $p$-adic valuation of the product of any Jacobi sum with its complex conjugate is 1. Use the elements of the kernel to find the additional relations that define the identity component of the Sato-Tate group.
\end{enumerate}
\end{algorithm}

Our work in Section \ref{sec:altmethod1}, Section \ref{sec:stickelberger}, Theorem \ref{theorem:matrix}, Theorem \ref{theorem:pointcountgeneral} and \ref{theorem:pointcount2} proves the following theorem.
\begin{theorem}\label{alg}
Algorithm (\ref{algorithm}) gives the identity component of the Sato-Tate group of curves of the form $$C_1\colon y^2=x^{2g+2}+c,\;  \; C_2\colon y^2=x^{2g+1}+cx, \; \; C_3\colon y^2=x^{2g+1} +c.$$
\end{theorem}

\subsection{Worked Examples}\label{sec:workedexamples}
We now use Algorithm \ref{algorithm} to prove Theorems \ref{genus4} and \ref{thm:x9}.

\begin{proof}[Alternate proof of Theorem \ref{genus4}]
We can use any prime $p\equiv 1 \pmod{10}$, so we choose to work in $ \mathbb{F}_{11}$ to simplify our calculations.  Theorem \ref{theorem:pointcountgeneral} tells us that the Jacobi sums that contribute are of the form $J(T_{10}^m,\phi)$, where $T_{10}=T^{(p-1)/10}$ and where $m$ ranges over all values from 1 to 9. The four embeddings from the group of characters to the unit circle  are given by $$T_{10} \rightarrow e^{\pi i/5} ,~~ T_{10} \rightarrow e^{3\pi i/5},~~T_{10}\rightarrow e^{7\pi i/5}, ~~ T_{10} \rightarrow e^{9\pi i/5}.$$ 
We compute the matrix described in Algorithm \ref{algorithm}. Its kernel is given by
  
$$\Span \left\{\begin{pmatrix} 1\\0\\0\\0\\-1\\0\\0\\0 \\1\\
\end{pmatrix},
\begin{pmatrix} 0\\1\\0\\0\\-1\\0\\0\\1\\0\\ 
\end{pmatrix},\begin{pmatrix} 0\\1\\0\\0\\-1\\0\\1\\0\\0\\
\end{pmatrix},\begin{pmatrix} 1\\0\\0\\0\\-1\\1\\0\\0\\0\\
\end{pmatrix},
\begin{pmatrix} -1\\0\\0\\1\\0\\0\\0\\0\\0\\
\end{pmatrix}
\begin{pmatrix} 0\\-1\\1\\0\\0\\0\\0\\0\\0\\
\end{pmatrix}\right\}.$$

Let $\chi_i=J(T^i_{10},\phi)$, so that each vector in the kernel is a tuple of exponents for the characters $$\chi_1, \chi_2, \chi_3, \chi_4,\chi_5, \chi_6, \chi_7, \chi_8, \chi_9.$$ The $p$-adic valuation of the product of any Jacobi sum with its complex conjugate is 1 and so, for example, $\chi_1\chi_9=1$. The additional relations are as follows. From the first vector,  $\chi_1\chi_5^{-1}\chi_9=1$ so $\chi_5=1$.  Similarly, from the second vector, $\chi_8=\chi_2^{-1}$. From the third vector, $\chi_7=\chi_2^{-1}$.   From the fourth vector, $\chi_6=\chi_1^{-1}$. From the fifth vector, $\chi_4=\chi_1$. Finally, from the sixth vector $\chi_3=\chi_2$. 
Thus,  
 $$\chi_1, \chi_2, \chi_3, \chi_4, \chi_6, \chi_7, \chi_8, \chi_9 = \chi_1,\chi_2,\chi_2,\chi_1,\chi_1^{-1},\chi_2^{-1},\chi_2^{-1},\chi_1^{-1}$$ and the identity component  of the Sato-Tate group of the Jacobian of $y^2=x^{10}+c$ is  $\U(1)_2 \times \U(1)_2.$
\end{proof}

\begin{theorem}\label{thm:x9b}
The identity component of the Sato-Tate group of the hyperelliptic curve $C\colon y^2=x^{9}+c$ is $\U(1) \times \U(1) \times \U(1).$
\end{theorem}
\begin{proof}
We can use any prime $p\equiv 1 \pmod{9}$, so we choose to work in $ \mathbb{F}_{19}$ to simplify our calculations.  Corollary \ref{coro:d=9} tells us that the Jacobi sums that contribute are $J(T_{9}^m,\phi)$, where $T_{9}=T^{(p-1)/9}$ and where $m$ ranges over all values from 1 to 8. Note that $T_9=T^{\frac{p-1}{9}}=T^2$, so we are only considering even powers of $T$. 

We have six embeddings into the circle, given by the primitive roots of unity  $e^{2\pi ik/9}$, where $\gcd(k,p-1)=1$. We compute the matrix $M$ described in Algorithm \ref{algorithm}. Its kernel is given by

$$\Span \left\{\begin{pmatrix} 1\\-1\\-1\\1\\0\\0\\0\\0
\end{pmatrix},
\begin{pmatrix}1\\-1\\0\\0\\-1\\1\\0\\0\\
\end{pmatrix},
\begin{pmatrix}1\\0\\-1\\0\\-1\\0\\1\\0\\
\end{pmatrix},
\begin{pmatrix}2\\-1\\-1\\0\\-1\\0\\0\\1\\
\end{pmatrix}
\right\}.$$
Let $\chi_i=J(T^i_{9},\phi)$, so that each vector in the kernel is a tuple of exponents for the characters $\chi_1, \chi_2, \chi_3, \chi_4, \chi_5, \chi_6, \chi_7, \chi_8$. The $p$-adic valuation of the product of any Jacobi sum with its complex conjugate is 1 and so, for example, $\chi_1\chi_8=1$. The additional relations are as follows. From the first vector, $\chi_4=\chi_2\chi_3\chi_1^{-1}$; from the second vector, $\chi_6=\chi_2\chi_5\chi_1^{-1}$; from the third vector,
$\chi_7=\chi_3\chi_5\chi_1^{-1}$; from the last vector,
$\chi_8=\chi_2\chi_3\chi_5\chi_1^{-2}$. 
Furthermore, since $\chi_1\chi_8=1$,  the relation from last vector can be written as $\chi_1=\chi_2\chi_3\chi_5$.   Substituting $\chi_1=\chi_2\chi_3\chi_5$ into $\chi_4=\chi_2\chi_3\chi_1^{-1}$ yields $\chi_5=\chi_4^{-1}.$ Repeating this process with the other relations yields  $\chi_6=\chi_3^{-1}$ and $\chi_7=\chi_2^{-1}$.   Hence, all characters can be written in terms of $\chi_2,\chi_3$, and $\chi_5$.

Putting this together we have  $$\chi_2,\chi_3,\chi_5,\chi_2^{-1},\chi_3^{-1},\chi_5^{-1}.$$ Thus, the identity component of the Sato-Tate group  of the Jacobian of $y^2=x^9+c$ is $$\U(1) \times \U(1) \times \U(1).$$

\end{proof}
As noted in the introduction, the identity component of the Sato-Tate group over $\mathbb{Q}$ is isomorphic to the Sato-Tate group over the CM field of the Jacobian of the curve.   To determine the Sato-Tate distribution, we need an explicit description of the embedding of the Sato-Tate of the Jacobian of the curve into $\USp(8)$ (see, for example, \cite[Remark 4.1]{Fite2016}).  Since $\chi_1=\chi_2\chi_3\chi_5$ and $\chi_8=\chi_1^{-1}$ we have the following embedding into $\USp(8)$
\begin{align*}
\U(1) \times \U(1) \times \U(1)& \simeq \langle \diag(u_1,\overline{u}_1,u_2,\overline{u}_2,u_3,\overline{u}_3)\rangle\\
&\simeq \langle \diag(u_1,\overline{u}_1,u_2,\overline{u}_2,u_3,\overline{u}_3,u_4(u_1,u_2,u_3),\overline{u_4(u_1,u_2,u_3)} )\rangle \\
&\simeq \ST^0(C_{\mathbb{Q}})\simeq \ST(C_F) \subseteq \USp(8)
\end{align*}
where we view $u_4$ as a function of $u_1,u_2,u_3$ and $F$ is the minimal extension over which all endomorphisms of the Jacobian of $y^2=x^9+c$ are defined.

\begin{corollary}\label{cor:c18}
The identity component of the Sato-Tate group of the Jacobian of the hyperelliptic curve $y^2=x^{18} +c$ is $\U(1)_2 \times \U(1)_2 \times \U(1)_2.$

\end{corollary}
\begin{proof}
Let $C\colon y^2=x^{18}+c$.  From Lemma \ref{evencase} 
$\Jac(C) \sim \Jac(C')$
where $C'\colon y^2=x^9+c$ and the result follows from Theorem \ref{thm:x9b}.  Alternatively, one can use Algorithm \ref{algorithm}.
\end{proof}

\subsection{Higher Genus Examples and Conjectures} \label{sec:highergenusex}
Using Algorithm \ref{algorithm} we compute additional examples of the identity component of the Sato-Tate group and formulate  conjectures for curves of the form  $$C_1\colon y^2=x^{2g+2}+c,  C_2\colon y^2=x^{2g+1}+cx,  C_3\colon y^2=x^{2g+1} +c$$
where $g$ is the genus of the curve and $c\in \mathbb{Q}^*$ is a constant.
As   previously stated, the   calculations for Algorithm \ref{algorithm} can be implemented in Sage \cite{Sage}.  Using Algorithm \ref{algorithm} we obtain the following.

\begin{table}[h]\caption{Identity components $\ST^0(C_1)$ for  genus 2 - 10.}
\begin{tabular}{c|l|l} 
    genus  of $C_1$& curve $C_1$&$\ST^0(C_1)$\\
    \hline 
    2& $y^2=x^6+c$&$\U(1)_2$\\
    3 & $y^2=x^8+c$&$\U(1)_2 \times \U(1)$\\
    4& $y^2=x^{10}+c$& $\U(1)_2 \times \U(1)_2$\\
    5 & $y^2=x^{12}+c$ &$ \U(1)_3 \times \U(1)_2$\\
    6& $y^2=x^{14}+c$& $\U(1)_2 \times \U(1)_2 \times \U(1)_2$\\
    7 & $y^2 =x^{16}+c $ & $\U(1)_2 \times \U(1)_2 \times \U(1)_2 \times \U(1)$\\
    8& $y^2=x^{18}+c$& $\U(1)_2 \times \U(1)_2 \times \U(1)_2$\\
    9 & $y^2 =x^{20}+c$ &$ \U(1)_4 \times \U(1)_2 \times \U(1)_2 \times \U(1)$\\ 
    10& $y^2=x^{22}+c$& $\U(1)_2 \times \U(1)_2 \times \U(1)_2 \times \U(1)_2 \times \U(1)_2$\\
    \end{tabular}
    
\end{table}

Note that the genus 2 example is also handled in \cite{Fite2012}, the genus 3 example is also handled in \cite [Cor. 5.3]{Fite2016}, and the genus 4 and 5 examples are worked out in Section \ref{sec:genus45proof} of this paper. This gives evidence for the following conjecture.
\begin{conjecture}
Let $C_{2p}\colon  y^2=x^{2p}+c$ where $p \geq 2$ is  prime.  Then 
$$\ST^0(C_{2p})=\underbrace{\U(1)_2\times \U(1)_2 \times \cdots \times \U(1)_2}_{(p-1)/2 \text{-times}}.$$
\end{conjecture}

We use Algorithm \ref{algorithm} again to compute the identity component of the Sato-Tate group for curves of the form $C_2\colon  \, y^2=x^{2g+1}+c$ and obtain the following.

\begin{table}[h]\caption{Identity components $\ST^0(C_2)$ for  genus 2 - 10.}
\begin{tabular}{c|l|l} 
    genus  of $C_2$& curve $C_2$&$\ST^0(C_2)$\\
    \hline 
    2 & $y^2=x^5+c$&$\U(1)^2$\\
    3 & $y^2=x^{7}+c$ &$ \U(1)^3$\\
    4 & $y^2 =x^{9}+c $ & $\U(1)^3$\\
    5 & $y^2 =x^{11}+c$ &$ \U(1)^5$\\ 
    6 & $y^2 =x^{13}+c $ & $\U(1)^6$\\
    7 & $y^2 =x^{15}+c$ & $\U(1)^4$\\ 
    8 & $y^2 =x^{17}+c$ &$  \U(1)^8$\\ 
    9& $y^2=x^{19}+c$&  $\U(1)^9$\\
    10& $y^2=x^{21}+c$ & $\U(1)^7$\\
\end{tabular}
    \end{table}
Note that the genus 2 example is also handled in \cite{Fite2012}. This gives evidence for the following conjecture.
\begin{conjecture}
Let  $C_p\colon  y^2=x^{p}+c$, where $p \geq 5$ is prime. Then 
$$\ST^0(C_p)=\U(1)^{(p-1)/2}.$$
\end{conjecture}
We also use Algorithm \ref{algorithm} to compute the identity component of the Sato-Tate group for curves of the form $C_3\colon  \,y^2=x^{2g+1}+cx$. We have the following results.
\begin{table}[h]    \caption{Identity components $\ST^0(C_3)$ for  genus 2 - 10.}
\begin{tabular}{c|l|l} 
    genus  of $C_3$& curve $C_3$&$\ST^0(C_3)$\\
    \hline 
    2 & $y^2=x^5+cx$&$\U(1)_2$\\
    3 & $y^2=x^{7}+cx$ &$ \U(1)_3$\\
    4 & $y^2 =x^{9}+cx $ & $\U(1)_2 \times \U(1)_2$\\
    5 & $y^2 =x^{11}+cx$ &$ \U(1)_5$\\ 
    6 & $y^2 =x^{13}+cx $ & $\U(1)_4 \times \U(1)_2$\\
    7 & $y^2 =x^{15}+cx$ &$ \U(1)_7$\\ 
    8 & $y^2 =x^{17}+cx$ &$  \U(1)_2 \times \U(1)_2 \times \U(1)_2 \times \U(1)_2$\\ 
    9& $y^2=x^{19}+cx$&  $\U(1)_9.$\\
    10&$y^2=x^{21}+cx$&$\U(1)_2 \times \U(1)_2 \times \U(1)_2 \times \U(1)_2$\\
\end{tabular}
\end{table}

Note that the genus 2 example is also handled in \cite{Fite2012} and the genus 3 example is also handled in \cite [Cor. 5.3]{Fite2016}. This gives evidence for our following conjecture.
\begin{conjecture}
Let  $C_3\colon  y^2=x^{2g+1}+cx$, where the genus $g=2k+1$ is odd. Then 
$$\ST^0(C_3)=\U(1)_g.$$
\end{conjecture}

\section*{Acknowledgments }  The authors would like to thank the Arizona Winter School and Andrew Sutherland for providing the research experience where this project began.   We give our heartfelt thanks to Francesc Fit\'e for his guidance and patience.    The authors also warmly thank Will Sawin for suggesting Algorithm \ref{algorithm} and for subsequent helpful discussions. We also thank Christelle Vincent, Holley Friedlander, and Fatma Cicek for their help with the computations in Section \ref{sec:highergenusex} during SageDays 103. We thank David Zywina for  discussing his results on the curve mentioned in Theorem \ref{thm:x9}. Finally, we thank the reviewer for their very thorough and helpful comments.

AP is supported by the Swiss National Science Foundation grant P2ELP2 172089. He was supported by the Simons Investigators Grant of Kannan Soundararajan during his time at Stanford University, when this paper was written.

\appendix

\section{Point Count Computation: $y^2=x^d+c$}
\label{sec:pointcount}

We have the following theorem regarding the point count of curves of the form $y^2=x^d+c$.

\begin{theorem}\label{theorem:pointcountgeneral}
Let $C_d\colon   y^2=x^{d}+c$ and let $p$ be an odd prime. Furthermore, let $T$ be a  fixed generator for the character group $\widehat{\mathbb F_p^{\times}}$ and $\phi=T^\frac{p-1}{2}$ be a quadratic character. Then
$$\#C_d(\mathbb F_p)  = p+1 + \sum_{m} \overline{T_d^m}(-c)\phi(c)J(\overline{T_d^m},\phi),$$
where the sum is over all $m\in\mathbb Z$ such that $\frac{(p-1)m}{d}\in [1, p-2]$ is integral, and $T_d^m=T^{\frac{m(p-1)}{d}}$.
\end{theorem}

\begin{remark}The number of terms in the point count formula partly depends on the congruence class of $p$. For example, if $p\equiv 1\pmod{d}$ then we will sum over the entire interval $[1,d-1]$. If $p-1$ and $d$ are relatively prime, then this summand will be empty. When $p-1$ and $d$ share at least some factors (for example, if $d$ is even) then there will be some terms that arise from this summand since we will be able to cancel the remaining factors in the denominator of $\frac{(p-1)m}{d}$ with some $m\in\mathbb Z$.

In the case where $p-1$ and $d$ are relatively prime, no $m\in\mathbb Z$ will yield a fraction $\frac{(p-1)m}{d}$ in the correct interval. To see why this is true, note that we would need $m=db$, for some positive $b\in\mathbb Z$, in order to have $\frac{(p-1)m}{d}\in\mathbb Z$. But this yields $$\tfrac{(p-1)m}{d}=(p-1)b\geq p-1,$$ which is not in the required interval.  Hence, in this case, the number of points will simply be $p+1$.
\end{remark}

\begin{proof}
Throughout, assume that $p$ is an odd prime. We follow the method of proof used in \cite{Fuselier10} to compute the number of points on a family of elliptic curves. Let $P(x,y)=x^{d}+c-y^2$. Recall from Section \ref{sec:gaussjacobi} that we define the additive character $\theta$ on $\mathbb F_p$ by $\theta(x)=\zeta^x$, where $\zeta$ is a primitive $p$-th root of unity. Since $\theta$ is an additive character $\theta(0)=1$, we have that
\[ \sum_{z\in\mathbb F_p} \theta(zP(x,y))= \begin{cases}
p& \text{if }P(x,y)=0,\\
0 &\text{otherwise.}
\end{cases}\]
Hence,
$$p\cdot(\#C_d(\mathbb F_p)-1) =  \sum_{z\in\mathbb F_p} \sum_{x,y\in\mathbb F_p} \theta(zP(x,y)).$$
Note that when $z=0$, $\sum_{x,y\in\mathbb F_p} \theta(0\cdot P(x,y))=p^2$. We break up the sum as follows
\begin{align*}
	 \sum_{x,y,z\in\mathbb F_p} \theta(zP(x,y)) &= p^2 +  \sum_{z\in\mathbb F_p^\times} \theta(zP(0,0))+ \sum_{y,z\in\mathbb F_p^\times} \theta(zP(0,y))\\&\hspace{.1in}+ \sum_{x,z\in\mathbb F_p^\times} \theta(zP(x,0))+ \sum_{x,y,z\in\mathbb F_p^\times} \theta(zP(x,y))\\
	 &:=p^2+A+B+C+D.
\end{align*}
 We will use Lemma \ref{lemma:charsum} and properties of Gauss sums  to evaluate each of these sums.

{\bf Computing A}:
\begin{align*}
 A=\sum_{z\in\mathbb F_p^\times} \theta(zP(0,0))&=  \sum_{z\in\mathbb F_p^\times} \theta(zc)\\
 								       &= \frac{1}{p-1} \sum_{z\in\mathbb F_p^\times} \sum_{i=0}^{p-2} G_{-i}T^i(zc)\\
								       &= \frac{1}{p-1} \sum_{i=0}^{p-2} G_{-i}T^i(c)\sum_{z\in\mathbb F_p^\times} T^i(z)\\
								       &=-1,
\end{align*}
since $\sum_{z\in\mathbb F_p^\times} T^i(z)=0$ unless $i=0$, in which case it equals $p-1$, and $G_0=-1$.\\

{\bf Computing B}:
\begin{align*}
B=\sum_{y,z\in\mathbb F_p^\times} \theta(zP(0,y))&=  \sum_{y,z\in\mathbb F_p^\times} \theta(zc)\theta(-zy^2)\\
			 &= \frac{1}{(p-1)^2} \sum_{i,j=0}^{p-2} G_{-i}G_{-j}T^i(c)T^j(-1)\sum_{z\in\mathbb F_p^\times} T^{i+j}(z)\sum_{y\in\mathbb F_p^\times} T^{2j}(y).
\end{align*}
Note that $\sum_{y\in\mathbb F_p^\times} T^{2j}(y)=0$ unless $j=0$ or $j=\frac{p-1}2$. In either case, $\sum_{z\in\mathbb F_p^\times} T^{i+j}(z)=0$ unless $i=j$. Hence, letting $\phi=T^{\frac{p-1}2}$,
\begin{align*}
B&=G_{0}G_{0}T^0(c)T^0(-1)+G_{\frac{p-1}2}G_{\frac{p-1}2}\phi(c)\phi(-1)\\
  &=1+p\phi(c),
\end{align*}
since $G_{\frac{p-1}2}G_{\frac{p-1}2} = p\phi(-1)$.\\

{\bf Computing C}:
\begin{align*}
C=\sum_{x,z\in\mathbb F_p^\times} \theta(zP(x,0))&=  \sum_{x,z\in\mathbb F_p^\times} \theta(zx^{d})\theta(zc)\\
			 &= \frac{1}{(p-1)^2} \sum_{i,j=0}^{p-2} G_{-i}G_{-j}T^j(c)\sum_{z\in\mathbb F_p^\times} T^{i+j}(z)\sum_{x\in\mathbb F_p^\times} T^{id}(x).
\end{align*}
We will not break this down further since this will cancel with part of sum $D$.\\

{\bf Computing D}:

\begin{align*}
D&=\sum_{x,y,z\in\mathbb F_p^\times} \theta(zP(x,y))\\
	&=  \sum_{x,y,z\in\mathbb F_p^\times} \theta(zx^{d})\theta(zc)\theta(-zy^2)\\
	&= \frac{1}{(p-1)^3} \sum_{i,j,k=0}^{p-2} G_{-i}G_{-j}G_{-k}T^j(c)T^k(-1)\sum_{z\in\mathbb F_p^\times} T^{i+j+k}(z)\sum_{x\in\mathbb F_p^\times} T^{id}(x)\sum_{y\in\mathbb F_p^\times} T^{2k}(y).
\end{align*}
As before, $\sum_{y\in\mathbb F_p^\times} T^{2k}(y)=0$ unless $k=0$ or $k=\frac{p-1}2$. Note that the case where $k=0$ negates the expression we found for $C$ since $G_{-k}=-1$ when $k=0$. We will denote the term with $k=\frac{p-1}2$ as $D'$. We break this term down further as follows.
\begin{align*}
D'&=\frac{1}{(p-1)^2} \sum_{i,j=0}^{p-2} G_{-i}G_{-j}G_{\frac{p-1}2}T^j(c)T^{\frac{p-1}2}(-1)\sum_{z\in\mathbb F_p^\times} T^{i+j+\frac{p-1}2}(z)\sum_{x\in\mathbb F_p^\times} T^{id}(x).
\end{align*}

The sum $\sum_{z\in\mathbb F_p^\times} T^{i+j+\frac{p-1}2}(z)=0$ unless $j=\frac{p-1}2-i$. Hence,
\begin{align*}
D'&=\frac{1}{p-1} \sum_{i=0}^{p-2} G_{-i}G_{-(\frac{p-1}2-i)}G_{\frac{p-1}2}T^{\frac{p-1}2-i}(c)T^{\frac{p-1}2}(-1)\sum_{x\in\mathbb F_p^\times} T^{id}(x).
\end{align*}

The sum $\sum_{x\in\mathbb F_p^\times} T^{id}(x)=0$ unless $i=0$ or $i$ is a multiple of $\frac{p-1}{d}$, i.e. $i=\frac{(p-1)m}{d}\in[0,p-2]$ for some $m\in\mathbb Z$. Hence,
\begin{align*}
D'&=G_{0}G_{\frac{p-1}2}G_{\frac{p-1}2}T^{\frac{p-1}2}(c)T^{\frac{p-1}2}(-1)\\&\hspace{.5in}+\sum_{m} G_{-m\frac{p-1}{d}}G_{m\frac{p-1}{d}-\frac{p-1}2}G_{\frac{p-1}2}T^{\frac{p-1}2-m\frac{p-1}{d}}(c)T^{\frac{p-1}2}(-1)\\
&=-p\phi(c)+\sum_{m} G_{-m\frac{p-1}{d}}G_{m\frac{p-1}{d}-\frac{p-1}2}G_{\frac{p-1}2}T^{-m\frac{p-1}{d}}(c)\phi(-c).
\end{align*}
 
Note that the term $-p\phi(c)$ will cancel with part of the expression in sum $B$. Letting $T_d^m =T^{\frac{m(p-1)}{d}}$  and recalling that $G_{a}:=g(T^a)$, we can write the above expression as
\begin{align*}
D'&=-p\phi(c)+\sum_{m} g(\overline{T_d^m})g(T_d^m\phi)g(\phi)\overline{T_d^m}(c)\phi(-c).
\end{align*}

Note that, for any nontrivial character  $A\not=\phi$, 
\begin{align*}
	g(\overline A)g(A\phi)g(\phi)&= g(\overline A)g(A\phi)g(\phi)\cdot \frac{g(\overline A\phi)}{g(\overline A\phi)}\\
						  &=\overline A\phi(-1)p	 \frac{g(\overline A)g(\phi)}{g(\overline A\phi)}\\
						  &=\overline A\phi(-1)pJ(\overline A,\phi),
\end{align*}
where the last equality holds by Equation \ref{eqn:jacobigauss}. On the other hand, if $A=\phi$, then 
\begin{align*}
	g(\overline A)g(A\phi)g(\phi)&= g(\phi)g(\epsilon)g(\phi)\\
						  &=-p\phi(-1)\\
						  &=\overline A\phi(-1)pJ(\overline A,\phi),
\end{align*}
where the last equality holds because $J(\phi,\phi)=-\phi(-1)$. Hence, for any nontrivial character $A$,
\begin{equation}\label{eqn:gaussjacobi}
g(\overline A)g(A\phi)g(\phi)=\overline A\phi(-1)pJ(\overline A,\phi).
\end{equation}
We use this to rewrite $D'$ as
\begin{align*}
      D'&=-p\phi(c)+p\sum_{m} \overline{T_d^m}(-c)\phi(c)J(\overline{T_d^m},\phi).
\end{align*}

We now combine these results to get
\begin{align*}
	\#C_d(\mathbb F_p) &= 1+ \frac1p\left(p^2+A+B+C+D\right)\\
				       &= 1+ \frac1p\left(p^2+p\sum_{m} \overline{T_d^m}(-c)\phi(c)J(\overline{T_d^m},\phi)\right)\\
				       &=p+1+\sum_{m} \overline{T_d^m}(-c)\phi(c)J(\overline{T_d^m},\phi),
\end{align*}
where  the sum is over all $m\in\mathbb Z$ such that $\frac{(p-1)m}{d}\in[1,p-2]$ is integral. 
\end{proof}

\begin{corollary}\label{coro:d=9}
The number of points on the curve $y^2=x^9+c$ over $\mathbb F_p$ is 
\[\#{\small C_9(\mathbb F_p)  = 
\begin{cases}
p+1 + \sum_{m=1}^{8} \overline{T_9^m}(-c)\phi(c)J(\overline{T_9^m},\phi) & \text{if  }  p\equiv 1 \pmod 9\\\\
p+1 + \overline{T_9^{3}}(-c)\phi(c)J(\overline{T_9^{3}},\phi)+\overline{T_9^{6}}(-c)\phi(c)J(\overline{T_9^{6}},\phi) & \text{if  }  p\equiv 4,7 \pmod 9\\\\
p+1 & \text{if } p\equiv 2\pmod 3 \text{ or } p=3.
\end{cases}}\]
\end{corollary}

\begin{proof}
For this result we are merely applying Theorem \ref{theorem:pointcountgeneral} to the case where $d=9$. We need to determine when $\frac{(p-1)m}{9}\in[1,p-2]$ is integral.

If $p\equiv 1 \pmod 9$, then any integer $m$ will make $\frac{p-1}{9}m$ integral. We restrict $m$ to the interval $[1,8]$ so that $\frac{p-1}{9}m\in[1,p-2]$.

On the other hand, if  $p\equiv 4,7 \pmod 9$, i.e. $p\equiv 1 \pmod 3 $ and $ p\not\equiv1\pmod 9$, then any integer of the form $m=3b$, where $b\in\mathbb Z$, will make $\frac{p-1}{3}\frac{m}{3}$ integral. We restrict $m$ to the interval $[1,8]$ so that $\frac{p-1}{3}\frac{m}{3}\in[1,p-2]$. Hence, only $m=3$ and $m=6$ will contribute to the point count sum.

Finally, if $p\equiv 2\pmod 3$ then no $m\in\mathbb Z$ will yield a fraction $\frac{(p-1)m}{d}$ in the correct interval.
\end{proof}

\section{Point Count Computation: $y^2=x^d+cx$}
\label{sec:pointcount2}

\begin{theorem}\label{theorem:pointcount2}
Let $C_d\colon   y^2=x^{d}+cx$ and let $p$ be an odd prime. Furthermore, let $T$ be a fixed generator for the character group $\widehat{\mathbb F_p^{\times}}$ and $\phi=T^\frac{p-1}{2}$ be a quadratic character. Then
$$\#(C_d(\mathbb F_p))  = p +1+\sum_{m} \overline{T_{d'}^{2m+1}}(-c)\phi(c)J(T_{d'}^{2m+1},\phi),$$
where the sum is over all $m\in\mathbb Z$ such that $\frac{(2m+1)(p-1)}{2(d-1)}\in[0,p-2]$ is integral (so that $T_{d'}^{2m+1}:=T^\frac{(2m+1)(p-1)}{2(d-1)}$ is a character).
\end{theorem}

\begin{proof}[Proof of Theorem \ref{theorem:pointcount2}]
Throughout, assume that $p$ is an odd prime. We follow the method of proof used in Appendix \ref{sec:pointcount}. Let $P(x,y)=x^{d}+cx-y^2$. As in Appendix \ref{sec:pointcount}, this yields
$$p\cdot(\#C_d(\mathbb F_p)-1) = 1+ \frac1p \sum_{z\in\mathbb F_p} \sum_{x,y\in\mathbb F_p} \theta(zP(x,y)).$$
Note that when $z=0$, $\sum_{x,y\in\mathbb F_p} \theta(0\cdot P(x,y))=p^2$. We break up the sum as follows
\begin{align*}
	 \sum_{x,y,z\in\mathbb F_p} \theta(zP(x,y)) &= p^2 +  \sum_{z\in\mathbb F_p^\times} \theta(zP(0,0))+ \sum_{y,z\in\mathbb F_p^\times} \theta(zP(0,y))\\&\hspace{.1in}+ \sum_{x,z\in\mathbb F_p^\times} \theta(zP(x,0))+ \sum_{x,y,z\in\mathbb F_p^\times} \theta(zP(x,y))\\
	 &:=p^2+A+B+C+D.
\end{align*}
 
We will use Lemma \ref{lemma:charsum} and properties of Gauss sums to evaluate each of these sums.\\

{\bf Computing A}:

Since $P(0,0)=0$, $A=\sum_{z\in\mathbb F_p^\times} \theta(zP(0,0))= p-1$.\\

{\bf Computing B}: 
\begin{align*}
    B	=\sum_{y,z\in\mathbb F_p^\times} \theta(zP(0,y))&=\sum_{y,z\in\mathbb F_p^\times}\theta(-zy^2)\\
	&=\frac{1}{p-1} \sum_{i=0}^{p-2} G_{-i}T^i(-1)\sum_{y\in\mathbb F_p^\times} T^{2i}(y)\sum_{z\in\mathbb F_p^\times} T^{i}(z).
\end{align*}
Note that $\sum_{z\in\mathbb F_p^\times} T^{i}(z)=0$ unless $i=0$, in which case both of the sums over $z$ and $y$ equal $p-1$. Hence
$$B=G_0T^0(-1)(p-1) = -(p-1).$$

{\bf Computing C}:
\begin{align*}
C=\sum_{x,z\in\mathbb F_p^\times} \theta(zP(x,0))&=  \sum_{x,z\in\mathbb F_p^\times} \theta(zx^{d})\theta(zcx)\\
			 &= \frac{1}{(p-1)^2} \sum_{i,j=0}^{p-2} G_{-i}G_{-j}T^j(c)\sum_{z\in\mathbb F_p^\times} T^{i+j}(z)\sum_{x\in\mathbb F_p^\times} T^{id+j}(x).
\end{align*}
We will not break this down further since this will cancel with part of sum $D$.\\

{\bf Computing D}:

\begin{align*}
D&=\sum_{x,y,z\in\mathbb F_p^\times} \theta(zP(x,y))\\
	&=  \sum_{x,y,z\in\mathbb F_p^\times} \theta(zx^{d})\theta(zcx)\theta(-zy^2)\\
	&= \frac{1}{(p-1)^3} \sum_{i,j,k=0}^{p-2} G_{-i}G_{-j}G_{-k}T^j(c)T^k(-1)\sum_{z\in\mathbb F_p^\times} T^{i+j+k}(z)\sum_{x\in\mathbb F_p^\times} T^{id+j}(x)\sum_{y\in\mathbb F_p^\times} T^{2k}(y).
\end{align*}
Note that  $\sum_{y\in\mathbb F_p^\times} T^{2k}(y)=0$ unless $k=0$ or $k=\frac{p-1}2$. The case where $k=0$ negates the expression we found for $C$ since $G_{-k}=-1$ when $k=0$. We will denote the term with $k=\frac{p-1}2$ as $D'$. We break this term down further as follows.
\begin{align*}
D'&=\frac{1}{(p-1)^2} \sum_{i,j=0}^{p-2} G_{-i}G_{-j}G_{\frac{p-1}2}T^j(c)T^{\frac{p-1}2}(-1)\sum_{z\in\mathbb F_p^\times} T^{i+j+\frac{p-1}2}(z)\sum_{x\in\mathbb F_p^\times} T^{id+j}(x).
\end{align*}

The sum $\sum_{z\in\mathbb F_p^\times} T^{i+j+\frac{p-1}2}(z)=0$ unless $j=\frac{p-1}2-i$. Hence,
\begin{align*}
D'&=\frac{1}{p-1} \sum_{i=0}^{p-2} G_{-i}G_{-(\frac{p-1}2-i)}G_{\frac{p-1}2}T^{\frac{p-1}2-i}(c)T^{\frac{p-1}2}(-1)\sum_{x\in\mathbb F_p^\times} T^{i(d-1)+\frac{p-1}2}(x).
\end{align*}

The sum $\sum_{x\in\mathbb F_p^\times} T^{i(d-1)+\frac{p-1}2}(x)=0$ unless $i(d-1)+\frac{p-1}2$ is a multiple of $p-1$. This occurs when $i$ is an odd multiple of $\frac{p-1}{2(d-1)}$, i.e. $i=\frac{(2m+1)(p-1)}{2(d-1)}$ for some $m\geq 0$ in $\mathbb Z$.

\begin{align*}
D'&=\sum_{m} G_{-\frac{(2m+1)(p-1)}{2(d-1)}}G_{\frac{(2m+1)(p-1)}{2(d-1)}-\frac{p-1}2}G_{\frac{p-1}2}T^{\frac{p-1}2-\frac{(2m+1)(p-1)}{2(d-1)}}(c)T^{\frac{p-1}2}(-1).
\end{align*}

Letting $T_{d'}^{2m+1} =T^{\frac{(2m+1)(p-1)}{2(d-1)}}$ and recalling that $G_{a}:=g(T^a)$, we can write the above expression as
\begin{align*}
D'&=\sum_{m} g(\overline{T_{d'}^{2m+1}})g(T_{d'}^{2m+1}\phi)g(\phi)\overline{T_{d'}^{2m+1}}(c)\phi(-c).
\end{align*}
We use Equation \eqref{eqn:gaussjacobi}  to rewrite $D'$ as
\begin{align*}
      D'&=p\sum_{m} \overline{T_{d'}^{2m+1}}(-c)\phi(c)J(\overline{T_{d'}^{2m+1}},\phi).
\end{align*}

We now combine these results to get 
\begin{align*}
	\#C_d(\mathbb F_p) &= 1+ \frac1p\left(p^2+A+B+C+D\right)\\
				       &= 1+ p+\sum_{m} \overline{T_{d'}^{2m+1}}(-c)\phi(c)J(T_{d'}^{2m+1},\phi)
\end{align*}
where, as above, the sum is over all $m\in\mathbb Z$ such that $\frac{(2m+1)(p-1)}{2(d-1)}\in[0,p-2]$ is integral.

\end{proof}

We will now explore the sets of $m$ that we obtain for different values of $p$ and $d$. Our sum of Gauss sums in the point count is over all $i$ such that $i(d-1)+\frac{p-1}{2}$ is an integer multiple of $p-1$. Hence, we are looking for values of $i$ that are odd multiples of $\frac{p-1}{2(d-1)}$.

First note that, regardless of the value of $d$, $2(d-1)$ is even. We can write $2(d-1)=2^ln$, for some odd $n\in \mathbb Z$. Hence, $\frac{(2m+1)(p-1)}{2(d-1)}\in[0,p-2]$ is integral when $(2m+1)(p-1)$ is divisible by $2^ln$. The values of $m$ will now depend on $p$. We split into cases.

If $p\equiv 1 \pmod{ 2^ln}$, then $i=\frac{(2m+1)(p-1)}{2(d-1)}=\frac{(2m+1)(p-1)}{2^ln}$ is an integer for any integer $m$. We restrict $m$ to the interval $[0,d-2]$ in order to obtain $i \in[0,p-2]$. To see why this is true, note that if $m=0$ then $\frac{(2m+1)(p-1)}{2(d-1)}=\frac{p-1}{2(d-1)}$, which is in the interval $[0,p-2]$. Similarly, if $m=d-2$, then $\frac{(2m+1)(p-1)}{2(d-1)}=\frac{(2d-3)(p-1)}{2d-2}<p-2$. \footnote{Note that this is true whenever $p>2d-1$. For smaller values of $p$, we will need to further restrict how large $m$ is.}  However if $m=d-1$, then $\frac{(2m+1)(p-1)}{2(d-1)}=\frac{(2d-1)(p-1)}{2d-2}>p-1>p-2$.

Suppose instead that $p\equiv 1 \pmod{ 2^ln'}$, where $n'<n$ is a divisor of $n$ (and that $p\not\equiv 1 \pmod{ 2^ln}$). In this case, $i=\frac{(2m+1)(p-1)}{2(d-1)}$ is an integer whenever $2m+1$ is a multiple of $n/n'$. To see why this is true, we let $2m+1=\frac{n}{n'}(2k +1)$, where $k$ is some integer. Then 
$$\frac{(2m+1)(p-1)}{2(d-1)}= \frac{\frac{n}{n'}(2k+1)(p-1)}{2^ln} = (2k+1)\cdot \frac{p-1}{2^ln'},$$
which is in $\mathbb Z$. We restrict $k$ to the interval $\left[0,\frac{2^ln'-1}{2}-1\right]$ in order to obtain $i \in[0,p-2]$ since if $k=\frac{2^ln'-1}{2}-1$ then
$$i=\frac{\frac{n}{n'}(2k+1)(p-1)}{2^ln} = \frac{(2\cdot\frac{2^ln'-1}{2}-2+1)(p-1)}{2^ln'} = \frac{(2^ln'-1)(p-1)}{2^ln'}<p-1.$$

Note that in the special case where $n'=1$, then we simply need $2m+1=n(2k+1)$ and $k\in\left[0,\frac{2^l-1}{2}-1\right]$.

Finally, suppose $p\not\equiv 1 \pmod{ 2^l}$. In this case, there are no values of $m$ such that $i=\frac{(2m+1)(p-1)}{2^ln}$ is an integer because we will be left with an even number in our denominator after canceling powers of 2 with $p-1$. In this case, the point count formula reduces to 
$$|C_d(\mathbb F_p)| = p+1.$$

We demonstrate this in the following example.

\begin{example}\label{example:g=3}
We will examine the number of points on the curve $y^2=x^7+cx$ over $\mathbb F_p$ for various primes $p$. Note that this is the genus 3 curve studied in \cite{Fite2016}. Since $d=7$, we have $2(d-1)=12$ and $d-2=5$.

If $p\equiv 1 \pmod{12}$, then we will have the maximum number of values for $i$. Explicitly, we have 
$$i\in\left\{\left. \frac{(2m+1)(p-1)}{12} \;\right| \; 0\leq m\leq 5  \right\}.$$
Thus, when $p\equiv 1 \pmod{12}$, our point count will be
\begin{align*}
	\#C_7(\mathbb F_p) &= p +1+ \sum_{m=0}^{5} \overline{T_{d'}^{2m+1}}(-c)\phi(c)J(\overline{T_{d'}^{2m+1}},\phi)
\end{align*}
where $T_{d'}$ is a character of order 12.\\

If $p\equiv 1 \pmod{4}$ and $p\not\equiv 1 \pmod{12}$, we will still have some terms from the Jacobi sum expression. Note that in this case $\frac{(2m+1)(p-1)}{12}$ will be an integer whenever $2m+1$ is divisible by 3. Hence, 
$$i\in\left\{\left. \frac{3(2k+1)(p-1)}{12} \;\right|\; 0\leq k\leq d/3-1  \right\}.$$
This yields the following
\begin{align*}
	\#C_7(\mathbb F_p)  &= p+1+ \overline{T_{d'}^3}(-c)\phi(c)J(\overline{T_{d'}^{3}},\phi)+ \overline{T_{d'}^9}(-c)\phi(c)J(\overline{T_{d'}^{9}},\phi)
\end{align*}
where $T_{d'}^3$ is a character of order 4.

If $p\equiv 3 \pmod{4}$, then $\frac{(2m+1)(p-1)}{12}$ will never be an integer. Hence, 
\begin{align*}
	\#C_7(\mathbb F_p) &= p+1.
\end{align*}

\end{example}

\bibliographystyle{abbrv}
\bibliography{AWSFinal2018}

\end{document}